\documentclass[12pt]{amsart}

\usepackage{amssymb}
\usepackage{url,amstext,amsfonts,amssymb,amscd,amsbsy,amsmath,amsthm,verbatim}
\usepackage{color}
\usepackage{amsthm}
\usepackage{latexsym}
\usepackage{enumerate}
\usepackage{hyperref}
\usepackage{graphicx}
\newtheorem{theorem}{Theorem}[section]
\newtheorem{defn0}[theorem]{Definition}
\newtheorem{prop0}[theorem]{Proposition}
\newtheorem{thm0}[theorem]{Theorem}
\newtheorem{lemma0}[theorem]{Lemma}

\theoremstyle{remark}
\newtheorem{remark}[theorem]{Remark}
\newtheorem{example}[theorem]{Example}
\theoremstyle{definition}
\newtheorem{algor0}{Algorithm}

\def\rig#1{\smash{ \mathop{\longrightarrow}
    \limits^{#1}}}

\def\dow#1{\Big\downarrow
   \rlap{$\vcenter{\hbox{$\scriptstyle#1$}}$}}

\def\O{{\mathcal O}}
\def\I{{\mathcal I}}

\def\C{{\mathbb C}}
\def\Z{{\mathbb Z}}
\def\P#1{{\mathbb P}^#1}

\newcommand{\propref}[1]{Prop.~\ref{#1}}
\newcommand{\thmref}[1]{Thm.~\ref{#1}}
\newcommand{\lemref}[1]{Lemma~\ref{#1}}

\newcommand{\exref}[1]{Example~\ref{#1}}
\newcommand{\secref}[1]{Section~\ref{#1}}
\newcommand{\rmkref}[1]{Rmk.~\ref{#1}}

\newcommand{\im}{\operatorname{im}}

\newcommand{\Hom}{\operatorname{Hom}}

\newcommand{\SL}{\operatorname{SL}}

\newcommand{\rank}{\operatorname{rank}}
\newcommand{\baseloc}{\operatorname{baseloc}}
\newcommand{\codim}{\operatorname{codim}}
\newcommand{\PP}{\mathbb{P}}
\newcommand{\CC}{\mathbb{C}}

\def\bw#1{{\textstyle\bigwedge^{\hspace{-.2em}#1}}}

\date{\today}
\begin{document}

\title{Eigenvectors of Tensors and Algorithms for Waring decomposition}
\thanks{The first author is supported by National Science Foundation grant Award No. 0853000: International Research Fellowship Program (IRFP). The second author is member of GNSAGA-INDAM.}

\author{Luke Oeding, Giorgio Ottaviani}\email{oeding@math.berkeley.edu, ottavian@math.unifi.it}
\address{Dipartimento di Matematica ``U. Dini'' \\
Universit\`a degli Studi di Firenze \\
Firenze, Italy}

\begin{abstract}
A Waring decomposition of a (homogeneous) polynomial $f$ is a minimal sum of powers of linear forms expressing $f$. Under certain conditions, such a decomposition is unique.
We discuss some algorithms to compute the Waring decomposition, which are linked to the equations of certain secant varieties and to eigenvectors of tensors.
In particular we explicitly decompose a cubic polynomial in three variables as the sum of five cubes (Sylvester Pentahedral Theorem).
\end {abstract}

\maketitle
\section{Introduction}
In this article we shall be concerned with the following general problem: given a polynomial, what is its minimal decomposition as a sum of powers of linear forms?

Let $V$ be a complex vector space of dimension $n+1$ and let $S^{d}V$ denote the space of $d^{th}$-order symmetric tensors.  A choice of basis $\{x_{0},\dots,x_{n}\}$  for $V$ induces a natural choice of basis for $S^{d}V$ and allows one to express $f\in S^{d}V$ as a polynomial in the variables $x_{i}$.
Because of {our} interest in tensor decomposition, we often do not make a distinction between a symmetric tensor and a polynomial, and often use the terms interchangeably. 

Let $f\in S^dV$. In this paper, our focus is on the symmetric tensor decomposition
\begin{equation}\label{eq:waring f}
f=\sum_{i=1}^rc_{i}(v_i)^d,
\end{equation}
where $v_i \in V$ have degree $1$ and $c_{i}\in \CC$. For historical reasons, we will refer to this as a {\it Waring decomposition} or simply a {\it decomposition} of $f$. 
 The minimum number of summands occurring in such a decomposition is called the {\it (symmetric) rank} of $f$.  
 \begin{remark}Note that when working over $\CC$, the constants $c_{i}$ in \eqref{eq:waring f} may be assumed to be equal to $1$, however this ambiguity will become important in our algorithms as initially we will only find the classes of the linear forms $[v_{i}] \in \PP V$ (the projective space of lines in $V$) and we will have to solve an easy linear system to find appropriate $c_{i}$ to resolve this indeterminacy in order to complete our decomposition algorithms.
\end{remark}

Because of the vast amount of work done in this area by several authors in diverse areas of science, the Waring decomposition is also known by many other names such as canonical decomposition (CANDECOMP/PARAFAC (CP)), rank-1 decomposition, sum of powers decomposition, and so on.

We start with an example that illustrates the main results of this paper, which are Algorithm~\ref{alg:metaalgorithm} and \thmref{thm:metathm}.  
A classical result, attributed to Hilbert, Richmond and Palatini, states that a general \footnote{see Remark~\ref{rmk:general} for a discussion of the term general} form $ f \in S^5\CC^3$ can be decomposed in unique way as the sum of seven fifth powers. For reference, see \cite{RanestadSchreyer} or \cite{Comas_Seiguer} and the references therein.

Set for the moment $V=\C^3$.  In the formula \eqref{eq:waring f} with $r=7, d=5$,
we want to find an efficient algorithm to find the $v_i\in\C^3$ and $c_i\in\C$ which give the decomposition of $ f $. To $ f $ we associate a linear map
\[P_{ f }\colon \Hom(S^2 V, V)\to \Hom(V, S^2V),\]
where $\Hom(S^2 V, V)$ is the space of linear maps from $S^2V$ to $V$.
Note that the target space is the dual of the source one.
If $ f =v^5$ then the definition of $P_{v^5}$ is the following
\begin{equation}
P_{v^5}(M)(w)=\left(M(v^2)\wedge v\wedge w\right)(v^2),
\end{equation}
where $M\in \Hom(S^2 V, V)$ and $w\in V$.
In general, if $ f $ is any element of $S^{5}V$, the definition of $P_{ f }$ is extended by linearity (which can always be done because $S^{d}V$ has a basis consisting of powers of linear forms).

To see the map $P_{ f }$ explicitly as a matrix, consider the monomial basis $\{x_ix_j\}_{0\le i\le j\le 2}$ of $S^2V$.
It turns out that with an appropriate ordering of the basis, $P_{ f }$ is represented by the following $18\times 18$  block matrix, where the nine depicted blocks have size $6\times 6$
\begin{equation}\label{eq:Pf}
\left[\begin{array}{rrr}
0& C_{f_2}&- C_{f_1}\\
- C_{f_2}&0& C_{f_0}\\
 C_{f_1}&- C_{f_{0}}&0\end{array}\right],
 \end{equation}
and $ C_{f_i}$ is the matrix whose entries labeled by $((j,k), (p,q))$ (for ${0\le j\le k\le 2}$, ${0\le p\le q\le 2}$) are
given by the fifth derivative $ f_{ikjpq}$.

{If $ f $ has rank $1$, we may choose coordinates such that
$f=x_0^5$. In this case $\rank C_{f_0}=\delta_{0i}$ and the matrix $P_{ x_{0}^{5} }$ has rank $2$. By the linearity of $P_{f}$ in the argument $f$ and the sub-additivity of matrix rank, if $ f $ has rank $r$, $P_{ f }$ must have rank $\leq 2r$.}  
In fact, we will see that in addition to giving this bound on rank, $P_{ f }$ will also be the key to an efficient way to decompose $ f $.

Now we have the following result (see Lemma~\ref{lem3.1}):
$P_{v^5}(M)=0$ if and only if there exists $\lambda$ such that $M(v^2)=\lambda v$.
Such $v$ is called an {\it eigenvector of the tensor} $M$ (or simply {\it eigenvector} of $M$, when the context is clear), 
where the homomorphism  $M\colon S^2V\to V$ is viewed as a tensor in $S^{2}V^{*}\otimes V$.

Eigenvectors of tensors were introduced and studied in \cite{Lim05,Qi05}. Recently, Cartwright and Sturmfels \cite{CartwrightSturmfels2011} have found a formula computing the number of eigenvectors of {a tensor $M \in \Hom(S^{m}\CC^{n},\CC^{n})$.}  (In  ~\secref{sec:Chern} we review their formula and propose a geometric interpretation of this formula that leads to an alternative proof which generalizes to other cases.)  The Cartwright-Sturmfels formula says that there are exactly {\it seven} eigenvectors $\{v_1,\ldots , v_7\}$ {of a general tensor in $\Hom(S^{2}\CC^{3},\CC^{3})$}. Later we will see that general elements of the kernel $\ker P_{ f }$ share the same eigenvectors.
{\it These seven eigenvectors appear exactly in the decomposition (\ref{eq:waring f}).} This is the basic novelty of this paper with respect
to \cite{LO11}, where the same example was considered. Our methods also provide a solution to the pentahedral example (see ~\secref{sec:pentahedral}), which was left open in
\cite{LO11}.

In our case, choose a basis $\{y_0, y_1, y_2\}$ of $V$ and its dual basis 
$\{x_0, x_1, x_2\}$ of $V^{*}$. 
Any $M\in S^{2}V^{*}\otimes V$ can be written as $\sum_{i=0}^2y_iq_i(x)$ where $q_i$ is a quadratic polynomial. {Recall that $v$ is an eigenvector of $M\in S^{2}V^{*}\otimes V$ if $M(v^{2})\wedge v =0$}, so the coordinates of the seven eigenvectors can be found
by the vanishing of the $2\times 2$ minors of the matrix
\begin{equation}\label{eq:2by2}
\left[\begin{array}{ccc}x_0&x_1&x_2\\
q_0&q_1&q_2\end{array}\right].
\end{equation}
Now we summarize the steps to compute the decomposition of general plane quintics, $ f \in S^5\CC^{3}$:

\begin{algor0}
\hfill\par
\textbf{Input:} $ f  \in S^{5}\CC^{3}$.
\begin{enumerate}
\item Construct the matrix $P_{ f }\colon \Hom(S^2 \CC^{3}, \CC^{3})\to \Hom(\CC^{3}, S^2\CC^{3})$ as in \eqref{eq:Pf}.
\item Compute $\ker P_{f}$. Choose a general $M\in \ker P_{f}$ and write $M =\sum_{i=0}^2y_iq_i(x)$ as above.
\item Find eigenvectors $\{v_1,\ldots , v_7\} \in \CC^{3}$ of $M$ via the zero-set of the $2\times 2$ minors of \eqref{eq:2by2}. 
\item Solve the linear system $ f =\sum_{i=1}^7c_iv_i^5$ in the unknowns $c_i\in \CC$.
\end{enumerate}
\par\textbf{Output:} The unique Waring decomposition of $f$.
\end{algor0}

In this article, our approach to Waring decomposition uses algebraic geometry, starting with the classical Sylvester algorithm and the notion of eigenvectors of tensors.  With the aid of recent progress \cite{LO11}, on equations of secant varieties using vector bundle techniques, we are able to go further. In fact, this paper can be seen as a constructive version for the symmetric case, of the techniques developed in  \cite{LO11}. 
The main result of this paper is a new algorithm for efficient Waring decomposition of symmetric tensors, stated in general in Algorithm~\ref{alg:metaalgorithm}.  This algorithm is a consequence of the geometric facts contained in\propref{prop:baseloc} and \thmref{thm:metathm}.  

Of course our algorithm will not always succeed to produce the Waring decomposition of symmetric tensors unless the rank is sufficiently small and the tensor is general. On the other hand, we can give precise bounds for the maximum rank of a symmetric tensor which can be decomposed via our algorithm.
In \thmref{thm:IKbound}, we give an improvement to the Iarrobino-Kanev bound on the applicability of the catalecticant method for Waring decomposition.
Going further, we state in Theorems~\ref{thm:k2bound} and~\ref{thm:knbound} sufficient conditions for the success of our algorithm in the cases of symmetric tensors on $3$ or more variables, respectively. We give a further discussion of what happens in the case that our algorithm fails in Remarks~\ref{rmk:failure1},~\ref{rmk:failure2}. In particular there is at least one case where our algorithm fails to decompose a tensor, but, as a side effect, brings to light a new way to {express} a rational quartic curve through 7 {(given)} general points, see Remark~\ref{rmk:rationalquartic}.

In addition to finding a new algorithm for Waring decomposition, and bounds for its success, we find a new proof of a result of Cartwright-Sturmfels (Proposition~\ref{prop:CS}) on the number of generalized eigenvectors that uses Chern classes and generalizes to other types of generalized eigenvectors (see\propref{prop:wedge_eigen}), that we call simply eigenvectors of tensors.


The use of tensors is widespread throughout science and appears in areas such as  Algebraic Statistics, Chemistry, Computer Science, Electrical Engineering, Neuroscience, Physics and Psychometrics.  A common theme at the 2010 conference on Tensor Decomposition and Applications in Monopoli, Italy was the need for efficient and reliable algorithms to perform tensor decomposition.
Indeed, we are certainly not the first to consider this problem. For a sample of related recent progress on tensor decomposition, see \cite{  BCMT, Comon-Mourrain, BallicoBernardi2010, BernardiGI, BucBuc}. Our aim is to use algebraic geometry as a basis for algorithms that can be used (either in place of or in combination with the previous algorithms) to improve efficiency and robustness.  We make some comparison with our methods and those of \cite{BCMT} in Remark~\ref{rmk:BCMT}.

Kolda and Mayo \cite{KoldaMayo} have recently studied an efficient way of computing eigenvectors of tensors, analogous to the usual iterative procedure to compute
usual eigenvectors of matrices.  Further work has been done by Ballard, Kolda and Plantenga \cite{BKP} to implement this method and they have achieved significant speed-ups utilizing a GPU when eigenvectors of many small tensors are to be computed.  Since our tensor decomposition methods use eigenvectors of tensors, this indicates that these methods could be combined with our algorithms to improve efficiency.

Another aspect of using tensor eigenvectors in our algorithms is that it may be reasonable to try (in the tensor setting) to mimic a method to approximate a matrix by one of lower rank via eliminating the eigenspaces corresponding to small eigenvalues. We hope that this article can serve as a starting point for such a study. For more issues regarding low-rank approximation of tensors, and the well-posedness of this problem, see \cite{deSilva_Lim}.

We have structured this article for two diverse audiences; algebraic geometers and researchers from a variety of applied fields studying tensors.  With algebraic geometers in mind, our goal is to show how well-known techniques in algebraic geometry can be used to solve problems in applied areas.
Most of our constructions are better explained with the geometric language of {\it vector bundles}.
However, we did our best to first state the main results by using explicit matrices, without the language of vector bundles. For this reason, we defer the main proofs until ~\secref{sec:bundle method}, in the sense that almost all our results are particular cases of a general result (\thmref{thm:metathm}), which is stated with the language of vector bundles.
Keeping in mind researchers studying applied tensor problems, in ~\secref{sec:classical} we describe the history and state of the art of algebraic geometry concerning tensor decomposition from a practical point of view. Because in their original versions the statements may not have been so accessible for applications, we have restated results from algebraic geometry, hopefully in a more transparent language, concerning generic rank (see \thmref{thm:AH}) and uniqueness of tensor decomposition (see \thmref{thm:uniqueness}).  In ~\secref{sec:Koszul}, we go on to illustrate our new techniques that use Koszul matrices to compute Waring decomposition. As mentioned above, in ~\secref{sec:bundle method} we give the generalization of our techniques. Finally, in ~\secref{sec:M2} we describe our Macaulay2 implementation of our algorithms, and we hope that this will serve as a starting point for further implementations of these methods.

\section{Classical methods for tensor decomposition: Sylvester's catalecticant method}\label{sec:classical}
\subsection{General results on the symmetric rank}\label{sec:rank}
\begin{remark}\label{rmk:general}
We use the term ``general element'' in the sense of algebraic geometry to indicate that the element is chosen to avoid a (Zariski) closed set.  
So when we say that the general element in a variety $X$
has a property, this means that $X$ contains a dense subset $X^0$
such that every element in $X^0$ satisfies that property.  We call a property ``generic'' if it holds for general elements.  
We consider the ``general'' assumption a mild assumption because in practice, most tensors we encounter in nature will actually be general. It is no loss to replace ``generic'' with ``almost always'' and ``general'' with ``randomly chosen,'' or ``up to certain non-degeneracy conditions.''
\end{remark}

{Regarding ranks of tensors, the following capstone theorem answers the question completely for general symmetric tensors.}
\begin{theorem}[\cite{AH2}]\label{thm:AH} 
Let $V$ be a complex vector space of dimension $n+1$.
The general $f\in S^dV$ has rank 
$$\left\lceil
\frac{{{n+d}\choose d}}{n+1}\right\rceil,$$ 
which is called the
{\it generic rank}, with the only exceptions
\begin{itemize}
\item{}$d=2$, where the generic rank is $n+1$.
\item{}$2\le n\le 4$, $d=4$, where the generic rank is 
${{n+2}\choose 2}$.
\item{}$(n,d)=(4,3)$, where the generic rank is $8$.
\end{itemize}

\end{theorem}
{Note that we use the functions $\lceil\cdot\rceil$ and  $\lfloor\cdot\rfloor$  respectively to indicate round-up and round-down.} 
The elements of rank one in $S^dV$ are just the polynomials that are $d$-th powers of a linear polynomial.
They form an irreducible algebraic variety, which is the cone
over a projective variety which is called the $d$-Veronese variety of $\PP^n$
and we denote $v_d(\PP^n)$.
 For all values of $k$ less than the generic rank, the (Zariski) closure of the elements of rank $\leq k$ is a irreducible variety,
which is the cone over $\sigma_k(v_d(\PP^n))$, the latter is called the $k$-th secant variety of $v_d(\PP^n)$. A consequence of the Alexander-Hirschowitz theorem
is that if 
$k\le \left\lfloor \frac{{{n+d}\choose d}}{n+1}\right\rfloor$, 
then $\dim\sigma_k(v_d(\PP^n))=k(n+1)-1$ {(the expected dimension)}
with the only exceptions 
\begin{itemize}
\item $d=2$, $2\le k\le n$
\item $2\le n\le 4$, $d=4$, $k={{n+2}\choose 2}-1$
\item $(n,d)=(4,3)$, $k=7$.
\end{itemize}
The cases listed above are called the {\it defective cases}.

{After the rank of a tensor is known, the next natural question is whether there is a unique decomposition (ignoring trivialities). The following represents the state of the art regarding this question. }
\begin{theorem}[\cite{CC02}, \cite{Mella06}, \cite{Ballico2005}]\label{thm:uniqueness}
For all values of $r$ smaller than the generic rank,
the general element of rank $r$ in $S^dV$ has a {\it unique (up to scaling)} decomposition $f=\sum_{i=1}^r c_{i}(v_i)^d$
with the only exceptions
\begin{enumerate}
\item the defective cases, where there are infinitely many decompositions
\item rank $9$ in $S^6\C^3$, where there are exactly two decompositions
\item rank $8$ in $S^4\C^4$, where there are exactly two decompositions.
\end{enumerate}
\end{theorem}

The cases listed as (2) and (3) in the above theorem are called the {\it weakly defective cases}.
For the generic rank, it is known that when $n+1$ does not divide ${{n+d}\choose d}$ then  there are infinitely many decompositions. On the other hand, when
$n+1$ divides ${{n+d}\choose d}$, apart from the defective cases, then there are finitely many decompositions.
In the latter case, it is expected that the decomposition of a general element in $S^dV$ is rarely unique.
{In this situation,} the only cases where uniqueness is known to hold are {the following:}
\begin{itemize}
\item{}$S^{d}\C^2$, for odd $d$, which was addressed by Sylvester in 1851.
\item{}$S^5\C^3$, which is the case addressed in the introduction.
\item{}$S^3\C^4$, the uniqueness result for the general rank, which is $5$, is known as the {\it Sylvester Pentahedral Theorem}.
\end{itemize} 
It is expected that these three are the only cases 
where uniqueness of decompositions hold for the general element.
Partial results confirming this expectation are proved in \cite{Mella09}.

A consequence of our construction is that we can treat {all three of these cases} in a unified manner.  The fact that our construction only finds these three cases gives further evidence that these may be the only exceptional cases for uniqueness.

Despite this beautiful theoretical picture, it is hard to compute  the rank of a given symmetric tensor and to find explicitly its tensor decomposition.
The brute force attempt to solve \eqref{eq:waring f} in the unknowns defining each $v_i$ with a computer algebra system is time and memory consuming and often fails, even in small dimension.  In fact, it is known that most tensor problems are extremely hard and often unsolvable \cite{Hillar09mosttensor}.  So while our goal is to make improvements in efficiency and reliability of tensor decomposition algorithms, we know that the generic problem for large tensors will remain difficult.

\subsection{The catalecticant method}\label{sec:cat}

The catalecticant method was developed in the XIX century, by Sylvester and others, to compute the rank of a symmetric tensor in $S^d\CC^{n+1}$ and to compute its Waring decomposition. The method {is completely successful} in the case of binary forms, \emph{i.e.} $n=1$,
but gains only partial success for $n\ge 2$. It is important to understand it deeply,
because most of successive methods proposed, including our method  developed in this paper, can be considered as a generalization of the catalecticant method.

Let $\{x_{i}\}_{1\leq i \leq n+1}$ be a basis of a vector space $V$. Given $f \in S^{d}V$ one can define maps for each $m< d$
\begin{equation}\label{eq:catdef}
\begin{array}{rrl}
C^{m}_{f}\colon& S^{m}V^{*} & \rig{} S^{d-m}V\\
         &x_{i_{1}}\cdots x_{i_{m}}  &\mapsto \frac{\partial^{m} f}{\partial x_{i_{1}}\cdots \partial x_{i_{m}}}
.
\end{array}
\end{equation}
Seen as a matrix, $C_{f}^{m}$ is known as a catalecticant or (semi-)Hankel matrix.  From (the more recent) point of view of tensors, $C_{f}^{m}$ can be thought of as a \emph{symmetric flattening} of a symmetric tensor, which is a symmetric version of a flattening of a tensor.  For this and other types of flattenings see \cite{LO11}.

Note that if $f=l^d$ then $\rank C^{m}_f=1$,
hence if $f$ has rank $r$ then $\rank C^{m}_f\le r$.  This gives lower bounds on the rank of $f$. But in fact, in this case, when $C_{f}^{m}$ is non-trivial we will be able to further exploit this construction to aid in decomposing $f$.

In their book \cite{IarrobinoKanev}, Iarrobino and Kanev described the key ingredients to the classical catalecticant method.  In particular they describe an algorithm to find a Waring decomposition.
We have implemented this approach, see ~\secref{sec:cat method}. Here is a summary of the algorithm.

\begin{algor0}[Catalecticant Algorithm]{\cite[5.4]{IarrobinoKanev}}\label{catalg}
\hfill\par
\textbf{Input:} $f \in S^{d}V$, where $\dim V = n+1$.
\begin{enumerate}
\item  Construct, via \eqref{eq:catdef}, the most square possible catalecticant $C^{m}_{f}= C_{f}$ with  $m = \left\lceil\frac{d}{2}\right\rceil$.
\item Compute $\ker C_{f}$. Note that $ \rank (f) \geq \rank (C_{f})$.
\item Find  the zero-set $Z'$ of the polynomials in $\ker C_{f}$. 
\begin{enumerate}
\item If $Z'$ is not given by finitely many reduced points, stop; this method fails. 
\item Else continue with $Z' =\{[v_1],\ldots , [v_s]\}$.
\end{enumerate}
\item Solve the linear system defined by $f=\sum_{i=1}^sc_iv_i^d$ in the unknowns $c_i$.
\end{enumerate}
\par\textbf{Output:} The unique Waring decomposition of $f$.
\end{algor0}

Algorithm~\ref{catalg} is a special case of Algorithm~\ref{alg:metaalgorithm} below, where we take $E  =\O(m)$ and $L=\O(d)$.
The following theorem gives a sufficient condition that guarantee its success. It can be seen as a slight improvement
(at least for $n\ge 3$) of the bound described in \cite{IarrobinoKanev}, see Theorems~4.10A and 4.10B.

\begin{theorem}\label{thm:IKbound}
Suppose $f=\sum_{i=1}^rv_i^{d}$ is a general form 
of rank $r$ in $S^{d}V$, let $z_i=[v_i]\in\PP(V)$ 
be the corresponding points and let $Z=\{z_1,\ldots, z_r\}$. 
Set $m = \left\lceil\frac{d}{2}\right\rceil$.
\begin{enumerate}
\item If $d$ is even and $r\le {{n+m}\choose n} -n-1$ or 
if $d$ is odd and $r \leq {{n+m-1}\choose n}$, 
then 
\begin{equation}
\ker C_f=I_{Z,m},
\end{equation}
 where $I_{Z,m} \subset S^{m}V^{*}$ denotes the subspace of polynomials of degree $m$ vanishing on $Z\subset V$. 
Moreover Algorithm~\ref{catalg} produces the unique Waring decomposition of $f$.
\item
Finally if $d$ is even, and $r = {{n+m}\choose n} -n$, then it is possible that $Z \subsetneq Z'$, where $Z'$ is obtained by Algorithm~\ref{catalg}.  But still when $n=2$, the algorithm will produce the unique minimal Waring decomposition. Further when $n\geq 3$, the algorithm will succeed after repeating step (4) finitely many times using subsets $Z''\subset Z'$ of size $\rank (C_{f})$. 
\end{enumerate}
\end{theorem}

{ As mentioned in the introduction, many of our statements in this section are consequences of more general results that are proved later in ~\secref{sec:bundle method}. While the reader may better understand the following arguments after the methods in ~\secref{sec:bundle method} are presented, we believe that it is worthwhile to anticipate their use now so that the simple pattern in this case may be seen.}

\begin{proof} 
First notice that by \thmref{thm:uniqueness}, we know that in this case we have a unique decomposition.  

Here we use the more general theory applied to this specific case.
Now we prove (1). 
By\propref{prop:baseloc}, {with $E= \O(m)$ and $L=\O(d)$,}  we have the inclusion $I_{Z,m} \subseteq \ker C_{f}$, and by the same Proposition, the equality holds if we can show that  the map $H^{0}(E^{*}\otimes L )\rig{} H^{0}(E\otimes L_{\mid Z})$ is surjective. For our particular choice of $E$ and $L$, we are considering the map $H^{0}(\O(-m+d)) \rig{}H^{0}(\O(-m+d)_{\mid Z}) $.
Let $m'= d-m = \left\lfloor\frac{d}{2}\right\rfloor$. So this amounts to showing that the natural evaluation map 
\begin{align*}
S^{m'}V^{*}\rig{}& \CC^{r}\\
f \longmapsto& (f(v_{1}),\ldots,f(v_{r}))
\end{align*}
 is surjective.
 But this is equivalent to the condition that $\codim (I_{Z,m'}) = r$. The preceding condition is satisfied because (by a basic dimension counting argument) $r$ general points give independent conditions on hypersurfaces of degree $m'$ when $r$ is bounded by ${{n+m'}\choose m'}$. It follows that $\rank C_{f}=r$ and we can apply \propref{prop:equality} and conclude that $I_{Z,m'} =\ker C_{f}$.

Next we will use the following theorem.
\begin{theorem}\cite[Theorem 2.6]{CC02}, Assume that $r\leq {{n+m}\choose m}-n-1$.
Let $X$ be an irreducible projective variety.
For $r\ge 3$, if every $(r-2)$-plane spanned by general points $x_1,\ldots, x_{r-1}$ 
also meets $X$ in an $r$-th point $x_r$ different from $x_1,\ldots, x_{r-1}$, then $X$ is contained in a linear subspace $L$, with $\codim_L(X)\le r-2$.
\end{theorem}

Let $Z' = \baseloc(\ker C_{f})$. We claim that $Z'=Z$, and the equality $I_{Z,m} =\ker C_{f}$ implies that the points $Z'
$ can be used to give the decomposition of $f$. 
Indeed, we just showed that $Z' = \baseloc(I_{Z,m})$, and the latter is equal to $Z$ by applying Theorem 2.6 in \cite{CC02}, stated above.

In the odd case, note that we will have $C_{f}\colon S^{m}V \rig{} S^{m-1}V$, and we always have ${{n+m-1}\choose n}\le {{n+m}\choose n} -n-1$.

For the proof of (2), now $Z'$ will be given by $m^{n}$ points, complete intersection of $n$ hypersurfaces of degree $m$. Note that for $n\geq 2$ we have ${{n+m}\choose m} -n\leq m^{n}$.  For $n=2$, the 
$m^{2}$ points impose independent conditions on the hypersurfaces (curves) of degree $2m$.  Indeed, from the sequence
{\[
0 \rig{} \O \rig{} \O(m)^{2} \rig{} \I_{Z'}(2m)  \rig{}0
,\]}
one shows that {$h^{1}(\I_{Z'}(2m))=0$. Denote $Z'=\{v_1,\ldots ,v_{m^2}\}$. The last vanishing implies that the powers
$(v_i)^{2m}$ for $i=1,\ldots, m^2$ are linearly independent.}  Then the linear system in step (4) in the algorithm has a unique {solution ,}
 which moreover only uses $\rank(C_{f})$ of the points.

For $n\geq 3$, since the $v_{i}^{d}$ are no longer independent, the linear system in step (4) of the algorithm will no longer have finitely many solutions using all the forms.  However we can overcome this problem by repeating step (4) with each subset of $Z'$ of size $\rank (C_{f})$ until we find the decomposition.

For the rest of the theorem, apply Algorithm~\ref{catalg} and refer to \thmref{thm:metathm}.
\end{proof} 

\subsection{Limits of the catalecticant method}\label{sec:cat limit}
{As in the proof of \thmref{thm:IKbound},} let $m'=\left\lfloor\frac d2\right\rfloor$. Another way to see that the catalecticant method 
can work for polynomials of rank $r$ only if $r< {{n+m'}\choose{ m'}}$, is the following.
 The maximum rank of any catalecticant matrix $C_f^{m}\colon S^{m}V^{*} \rig{}S^{m'}V$ is ${{n+m'}\choose {m'}}$.  If $f$ has this rank or greater, we will either (in the even case) have no kernel to work with, or (in the odd case) fail to satisfy the equality $I_{Z,m'}=\ker C_{f}^{m}$.
 
{Usually the general rank} $\frac{{{n+d}\choose d}}{n+1}$ is larger than ${{n+m'}\choose {m'}}$, the first example
being $v_3(\PP^2)$, where the catalecticant matrices have size $3\times 6$ or $6\times 3$.
The equation of $\sigma_3(v_3(\PP^2))$, which is called the Aronhold invariant,
cannot be found as a minor of any catalecticant matrix.
It can be obtained as a Pfaffian of a Koszul flattening (see \cite{Ott2}).  In practice, one should start by applying the catalecticant method, and if it fails to produce the Waring decomposition, this implies that the rank is larger than the bounds listed in \thmref{thm:IKbound}. Next we introduce new algorithms that will succeed to decompose tensors in a larger range of ranks.

\section{New methods for tensor decomposition: Koszul flattening and eigenvectors of tensors}\label{sec:Koszul}
Note that the following definition also appeared in \cite{Ott2} in a few special cases and was further developed in \cite{LO11}, both with a focus of finding equations of secant varieties. We also note that this construction was used in the case of partially symmetric tensors to find the ideal-theoretic defining equations of the $k$-th secant variety (with $k\leq 5$) of $Seg(\PP^{2}\times \PP^{n})$ embedded by $\O(1,2)$ in \cite{CEO}.
Here our presentation is focused on using this construction to find decompositions of tensors via eigenvectors. 

The general setting of this section is the following.
Let $ f \in S^dV$ and fix $0\le a\le n$, $1\le m\le d-1$.
We construct a linear map
\begin{equation}\label{eq:pfdef1}
P_{ f }\colon \Hom(S^m V, \bw aV)\to \Hom(\bw {n-a}V, S^{d-m-1}V)
.\end{equation}
If $ f =v^d$ then the definition of $P_{v^d}$ is the following
\begin{equation}\label{eq:pfdef2}
P_{v^d}(M)(w)=\left(M(v^m)\wedge v\wedge w\right)(v^{d-m-1}),
\end{equation}
where $M\in \Hom(S^m V, \bw aV)$, $w\in \bw {n-a}V$ {and we fixed an isomorphism $\bw{n+1}V\simeq\C$}. 
If $ f $ any element of $S^{d}V$, the definition of $P_{ f }$ is extended by linearity.
{This extension is guaranteed by the fact that  on general decomposable tensors $f=v_1\otimes\ldots\otimes v_d$
the map $P_f$ has the following expression 
$$
P_{v_1\otimes\ldots\otimes v_d}(M)(w)=\sum_{\sigma}\left(M(v_{\sigma(1)}\otimes\ldots \otimes v_{\sigma(m)})\wedge v_{\sigma(m+1)}\wedge w\right)
(v_{\sigma(m+2)}\otimes\ldots\otimes v_{\sigma(d)})
,$$
which is visibly linear, where the summation is performed for $\sigma$ in the symmetric group of permutations on $d$ elements.
}

Although this definition might seem artificial at first glance, we now explain how it can be used. We wait until ~\secref{sec:presentations} for a more formal treatment of $P_{f}$ via a presentation of a vector bundle. The linear map $P_{f}$ can be explicitly computed by using Koszul matrices, which motivates the name \emph{Koszul flattening} that we give to $P_{f}$ and is intended to mirror the term symmetric flattening which it generalizes.  

In order to explicitly write down the matrix representing $P_{ f }$,
we need to recall the properties of the {\it Koszul complex}. It is the minimal resolution of
the field $\CC$ as an $R=\CC[x_0,\ldots , x_n]$-module.
Here we give some examples, but interested readers unfamiliar with the Koszul complex may wish to consult \cite{Eisenbud_syzygies}.
For $n=2$ the Koszul complex is
$$0\rig{}R(-3)\rig{k_3}R(-2)^3\rig{k_2}R(-1)^3\rig{k_1}R\rig{}\CC\rig{}0,$$
where
$$k_1=\bgroup\begin{pmatrix}
x_{0}& x_{1}& x_{2}\\
\end{pmatrix}\egroup$$
$$k_2=
\bgroup\begin{pmatrix}
0&{x_{2}}& {-x_{1}}\\
{-x_{2}}& 0& {x_{0}}\\
x_{1}& x_{0}&0\\
\end{pmatrix}\egroup.$$
For $n=3$ it is
$$0\rig{}R(-4)\rig{k_4}R(-3)^4\rig{k_3}R(-2)^6\rig{k_2}R(-1)^4\rig{k_1}R\rig{}\CC\rig{}0,$$
where
$$k_1=\bgroup\begin{pmatrix}
x_{0}& x_{1}& x_{2}& x_{3}\\
\end{pmatrix}\egroup$$
$$k_2=\bgroup\begin{pmatrix}
{-x_{1}}& {-x_{2}}& 0& {-x_{3}}& 0& 0\\
x_{0}& 0& {-x_{2}}& 0& {-x_{3}}& 0\\
0& x_{0}& x_{1}& 0& 0& {-x_{3}}\\
0& 0& 0& x_{0}& x_{1}& x_{2}\\
\end{pmatrix}\egroup.$$

In general we have
$$k_i\colon R(-i)^{\binom{n+1}{i}}\rig{}R(-i+1)^{\binom{n+1} {i-1}}.$$
The matrix $k_i$ is a Koszul matrix. It corresponds to the presentation of the vector bundle 
$\bw {n+1-i}Q (-i)$ and on the point $\langle v\rangle\in\PP V$ it corresponds to the wedge product 
$\bw {n+1-i}V\rig{\wedge v}\bw {n+2-i}V$.  We note that the Koszul complexes can be easily computed by any standard computational algebra system such as Macaulay2 \cite{M2}.

Next we recall a result from \cite{LO11} that gives an explicit version of this construction. 

\begin{lemma0}
Let $ f \in S^dV$ .
The matrix $P_{ f }\colon \Hom(S^mV,\bw aV)\to \Hom(\bw {n-a}V, S^{d-m-1}V)$
can be computed using the matrix $k_{n+1-a}$  of the Koszul complex,
of size ${{n+1}\choose {a}}\times {{n+1}\choose {a+1}}$, where at the place
of the indeterminate $x_i$ we substitute the catalecticant matrix $C_{ f_i}^m$ of size
${{n+d-m-1}\choose n}\times {{n+m}\choose n}$,
where $ f_i=\frac{\partial f }{\partial x_i}$.
The matrix $P_{ f }$ obtained has size 
$$ \left[{{n+m}\choose n}{{n+1}\choose a}\right]\times \left[{{n+d-m-1}\choose n}{{n+1}\choose {a+1}}\right].$$
\end{lemma0}
\begin{proof} See \cite[Section~8.3]{LO11}.
\end{proof}

Now we propose the following
\begin{defn0}
Given $M \in \Hom(S^{m}V, \bw{a} V)$, a vector $v\in V$ is called an \emph{eigenvector  
of the tensor} $M$ if
 \[M(v^{m})\wedge v = 0.\]
\end{defn0}
For $m=a=1$ this is a usual eigenvector, {and for $a=1$, any $m$ this agrees with the notion of \cite{Lim05, Qi05}.}

Now we have the following important lemma (whose proof is straightforward) that we would like to emphasize.
\begin{lemma0}\label{lem3.1}
Let $M\in \Hom(S^{m}V,\bw{a} V)$.
\begin{enumerate}
\item A vector $v\in V$ is an eigenvector of $M$ if and only if $M \in \ker(P_{v^{d}})$.
\item Let $ f =\sum v_i^d$.  If each $v_{i}$ is an eigenvector of $M$, then $M\in\ker P_{ f }$.
\end{enumerate}
\end{lemma0}

Cartwright and Sturmfels have recently found a formula computing the number of eigenvectors of a general tensor $M\in \Hom(S^mV,V)$, which is the case $a=1$ in our construction.
The following theorem generalizes the Cartwright-Sturmfels formula to any $a$.
\begin{theorem}\label{thm:counting}
For a general $M\in \Hom(S^mV, \bw aV)$, the number of $[v]\in \PP V$ such that $M(v^m)\wedge v=0$ is given by
\begin{itemize}
\item{} $\begin{cases} m \text{ for } a=0,2 \text{ and } n=1,\\
                     \infty \text{ when }a=0,n+1 \text{ and } n>1.\end{cases}$
\item{}$\frac{m^{n+1}-1}{m-1}$ for $a=1$
\item{}$0$ for $2\le a\le n-2$
\item{}$\frac{(m+1)^{n+1}+(-1)^n}{m+2}$ for $a=n-1$.
\end{itemize}
\end{theorem}
\begin{proof} The non-trivial cases follow from Propositions~\ref{prop:CS} and~\ref{prop:wedge_eigen}.
\end{proof}
For $a=0$ the condition becomes $M(v_i^m)=0$ and $P_{ f }$ reduces to the catalecticant.
The cases $a=1$ and $a=n-1$ are a bit special. Note that Hilbert's quintic example from the introduction fits the case $a=1$ (and agrees with the case $a=n-1$ since $n=2$), while we will see that Sylvester's pentahedral example fits in the case $a=n-1$.

As we mentioned in the introduction, in the case $a=1$, the iterative methods of \cite{KoldaMayo} may be used to find eigenvectors of tensors in $\Hom(S^{m}V,V)$, however for general $a$, more work needs to be done in order to efficiently find eigenvectors of tensors in $\Hom(S^{m}V,\bw{a}V)$.
The idea, like in the matrix case, is to iterate the map
$$v_{k+1}=\frac{M(v_{k}^{m})}{\|M(v_{k}^{m})\|},$$
and to successively approximate the eigenvectors, starting from the dominant one and repeating until all eigenvectors are found.

Next we describe a general algorithm to decompose polynomials, which corresponds to the case $a=1$, any $m$.  This algorithm is especially effective in the case $n = 2$.
\begin{algor0}[Koszul Flattening Algorithm 1]\label{alg:koszul}
\hfill\par
\textbf{Input:} $f\in S^{d}V$, where $\dim V = n+1$.
\begin{enumerate}
\item Construct, via \eqref{eq:pfdef1} and \eqref{eq:pfdef2}, $P_{ f }\colon \Hom(S^{m}V,V)\to \Hom(\bw {n-1}V, S^{d-m-1}V)$.
\item Compute $\ker P_{f}$ and note that $\rank (f)\geq \frac{\rank (P_{f})}{n}$.
\item Find $Z'$, the common (projective) eigenvectors of a basis of the kernel of $P_{f}$.
\begin{enumerate}
\item If $Z'$ is not given by finitely many reduced points, stop; this method fails. 
\item Else continue with $Z' =\{[v_1],\ldots , [v_s]\}$.
\end{enumerate}
\item Solve the linear system on the constants $c_{i}$ defined by setting $f =\sum_{i=1}^s c_iv_i^d$.
\end{enumerate}
\par\textbf{Output:} The unique Waring decomposition of $f$.
\end{algor0}

Now we state sufficient conditions in the case $n=2$ and $d$ is odd for Algorithm~\ref{alg:koszul} to succeed.  (In the case $d$ is even we don't state conditions because our experiments show that the Catalecticant Algorithm covers all cases that the Koszul algorithm can cover).

\begin{theorem}\label{thm:k2bound} Suppose $n=2$ and set $d=2m+1$. Let $f=\sum_{i=1}^rv_i^{d}$ be a general form of rank $r$ in $S^{d}V$, {let $z_i=[v_i]\in\PP(V)$ 
be the corresponding points and let $Z=\{z_1,\ldots, z_r\}$.}
Let $Z'$ be the set of common eigenvectors (up to scalars) of $\ker P_f$.
\begin{enumerate}
\item If $2r\le m^{2}+3m + 4$ then $Z' =Z$.
Moreover Algorithm~\ref{alg:koszul} produces the unique Waring decomposition of $f$.
\item If $2r \leq m^{2} + 4m +2$, then it is possible that $Z \subsetneq Z'$. Even in this case, Algorithm~\ref{alg:koszul} will produce the unique minimal Waring decomposition. 
\end{enumerate}
\end{theorem}
We postpone the proof of \thmref{thm:k2bound} until we have the tools from the next section.

Before going on, we note that the map $P_{f}$ factors, and therefore it will have smaller rank than what may be first expected. In order to accurately explain this, we will need to use representation theory and the language of partitions and Young diagrams according to \cite{FultonHarris}. 

The map $P_{f}$ always has a non-trivial kernel, which comes from an analogy to the matrix case and the fact that every vector is an eigenvector of a scalar multiple of the identity.  First we notice that $\Hom(S^{m}V ,\bw{a} V)$ splits as the direct sum of two $\SL(V)$-modules, {via the Pieri rule, (see \cite[equation (6.9), p.79]{FultonHarris})}
\begin{equation}\label{eq:trace}
\Hom(S^mV,\bw{a} V)=\Gamma^{m^n}\otimes\bw{a} V=\Gamma^{(m+1)^{a},m^{n-a}}\oplus \Gamma^{(m)^{a-1},(m-1)^{n-a+1}},
\end{equation}
where in general for any partition $\pi$, $\Gamma^{\pi}$ is the $\SL(V)$-representation associated to $\pi$. In partitions, we use the notation $m^{i}$ to denote $m$ repeated $i$ times. Note that $\Gamma^{\pi}$ inherently depends on $V$, but we suppress this from the notation for simplicity.

\begin{lemma0}
$P_{f}$ restricted to $\Gamma^{(m)^{a-1},(m-1)^{n-a+1}}$ is zero.
\end{lemma0}
\begin{proof}
By the linearity of $P_{f}$ in $f$, it suffices to prove the lemma for $f= v^{d}$.
The essential fact that we use is that the representation $\Gamma^{m^{a-1},(m-1)^{n-a+1}}$ is isomorphic to a subspace of $\Hom(S^{m-1}V, \bw{a-1} V)$, which indeed splits as 
\[\Hom(S^{m-1}V, \bw{a-1} V)= \Gamma^{m^{a-1},(m-1)^{n-a+1}}\oplus \Gamma^{(m-1)^{a-2},(m-2)^{n-a+2}}.\]

So consider $M \in \Gamma^{m^{a-1},(m-1)^{n-a+1}} \subset \Hom(S^mV,\bw{a} V)$ .  {There is a natural equivariant map
$$
\begin{array}{ccc}
\Hom(S^{m-1}V,\bw{a-1}V)&\to &\Hom(S^{m}V,\bw{a}V)\\
N&\mapsto&\left[v^m\mapsto  N(v^{m-1})\wedge v\right]\\
\end{array}
$$
which is nonzero, so by Schur lemma it identifies the summands $\Gamma^{m^{a-1},(m-1)^{n-a+1}}$ in both sides.
We write $\tilde M$ as the  copy of $M$ in $\Hom(S^{m-1}V, \bw{a-1} V)$ according to this identification.

Then $ M (v^{m}) = \tilde M(v^{m-1})\wedge v$,} so for this $M$, 
\[P_{v^d}(M)(w)=\left(M(v^m)\wedge v\wedge w\right)(v^{d-m-1})= \left(\tilde M(v^{m-1})\wedge v\wedge v\wedge w\right)(v^{d-m-1})  =0.\qedhere\]
\end{proof}

\subsection{The quotient bundle and eigenvectors of tensors}\label{sec:tangent bundle}
We show in this section that the eigenvectors of a tensor can be interpreted as zero loci of sections of twists of the quotient bundle. Similar methods were recently used in \cite{OttSturm} to study matrices with eigenvectors in a given subspace. For generalities about vector bundles we refer to \cite{OSS}.
$\PP V$ is the projective space of lines {in} $V$, therefore $H^0(\PP V, \O(1))=V^*$.
The quotient bundle of $\PP V$ which we will denote by $Q$, appears in the Euler exact sequence
\[0\rig{}\O_{\PP V}(-1)\rig{}\O_{\PP V}\otimes V\rig{}Q\rig{}0,\]
taking wedge powers and tensoring by $\O(m)$ we get the sequence
\[
\bw{a-1} V\otimes \O(m-1)\rig{} \bw{a}V \otimes \O(m)\rig{}\bw{a}Q(m)\rig{}0.
\]
Taking the global sections we get
\begin{equation}\label{eq:euler}
\Hom(S^{m-1}V,\bw{a-1}V) \rig{}\Hom(S^mV,\bw{a} V)\rig{\phi} H^0\left(\bw{a}Q(m)\right)\rig{}0.\end{equation}

For any tensor $M\in\Hom (S^mV,\bw{a}V)$ we denote by $s_M$ the section of $\bw{a}Q(m)$ corresponding to $\phi(M)$ in sequence \eqref{eq:euler}.
We want to show that the eigenvectors of the tensor $M$ correspond to the zero locus of $s_M$.
To make this construction precise, we recall the following straightforward lemma.
\begin{lemma0}\label{lem:quotient}
\hfill\par
\begin{enumerate}
\item The fiber of $\bw{a}Q(m)$ at $x=\langle v\rangle$ is isomorphic to $\Hom(\langle v^m\rangle,\bw{a}V/\langle v\wedge \bw{a-1}V\rangle)$.
\item
The section $s_M$ vanishes in $\langle v\rangle$ if and only if 
$v$ is an eigenvector of the tensor $M$.
\end{enumerate}
\end{lemma0}
\begin{proof}
The section $s_M\in H^0(\bw{a}Q(m))$ corresponds on the fiber of $v$ to the composition
$\langle v^m\rangle\rig{i}S^mV\rig{M} \bw{a} V\rig{\pi} \bw{a} V/\langle v\wedge \bw{a-1} V\rangle$
where $i$ is the inclusion and $\pi$ is the quotient map.
Now $s_M$ vanishes on $\langle v\rangle$ if and only if $\pi(M(v^m))=0$ if and only if $M(v^m)\wedge v =0$.
\end{proof}
\begin{remark}
Note that in the decomposition of formula \eqref{eq:trace}, 
$\Hom(S^mV,V)=\Gamma^{m^n}\otimes V=\Gamma^{m+1,m^{n-1}}\oplus S^{m-1}V^{*}$, we have 
from Bott's theorem that $H^0(Q(m))=\Gamma^{m+1,m^{n-1}}$. 
\end{remark}

Now we have the proper tools to anticipate the proof of \thmref{thm:k2bound}, 
even if we need some results of Sections~\ref{sec:bundle method}
 and~\ref{sec:sufficient conditions}.
\begin{proof}[Proof of \thmref{thm:k2bound}]
First notice that by \thmref{thm:uniqueness}, we know that in this case we have a unique decomposition.  

Now we prove (1) by applying \propref{prop:baseloc}.  To match with the notation of \propref{prop:baseloc}, take  $E = Q(m)$, a twist of the quotient bundle on $\PP^{2}$, and $L = \O(d)$ with $d=2m+1$. (See also \exref{ex:pfexample} for more details on the matrix version of this construction.)  By \propref{prop:baseloc}, we have the inclusion 
$H^{0}(Q\otimes \I_{Z}(m)) \subseteq \ker  A_{f}$.
Note that $E^{*}\otimes L=Q^{*}(-m+2m+1)=Q^{*}(m+1)=Q(m)=E$.
The  equality holds if the  map
$H^{0}(Q(m)) \rig{} H^{0}(Q_{Z}(m))$,
is surjective.
This is true by  \thmref{thm:applicabilityP2}(i).

It follows that $\rank A_f=\rank P_{f}=2r$ and we can apply \thmref{thm:metathm}. 
Moreover, in this case, $Z'=Z$. Indeed, we just showed that $Z'$ is the base locus of $H^{0}(Q\otimes \I_{Z}(m))$, and the latter is equal to $Z$ because of (ii) of \thmref{thm:applicabilityP2}.

For the proof of (2), now $Z'$ is contained in the zero-locus $Z''$ of a section of $Q(m)$.
It is enough to show that $Z''$ imposes independent conditions on the hypersurfaces of degree $2m+1$.
Indeed, from the sequence
\[
0 \rig{} \O \rig{} Q^{*}(m+1) \rig{} \I_{Z''}(2m+1)  \rig{}0
,\]
and the vanishing $h^1(Q^{*}(m+1))=0$, \cite{OSS} Ch.1~\S~1, reading $Q^{*}(m+1)=\Omega^1(m+2)$, one shows that $h^{1}(\I_{Z''}(2m+1))=0$.  Then the linear system in step (4) in Algorithm~\ref{alg:koszul} has a unique solution (because the powers $v_{i}^{d}$ are linearly independent), which moreover only uses $\rank (P_{f})/2$ of the points.
\end{proof}

\subsection{Pentahedral example}\label{sec:pentahedral}

Here we treat the Sylvester pentahedral example, in an analogous way to the Hilbert quintic case treated in the introduction.
Let $ f \in S^3\CC^4$.
The classical approach to the Pentahedral Theorem is to use
the $10$ points {$p_i$} which are the singular locus of the Hessian of $ f $.
These ten points are the vertices of the pentahedral formed by the five planes $v_i$.
We do not know how to express the five planes rationally from $ f $ with this approach.
Enriques and Chisini, in the third book of their textbook
 \cite{Enriques_book}, attribute to Gordan the computation of
 the fifth degree covariant of $ f $ given by the five planes, in terms of symbolic calculus,
but we found difficult to explicitly compute the Gordan covariant.

Our approach covers both the proof of the theorem and the possibility
to find explicitly the five planes.

\begin{thm0}{\bf Sylvester Pentahedral Theorem}
For any general $ f \in S^3\CC^4$, there exist unique $v_i\in \CC^4$ (up to scalar)
and $c_i\in\CC$ for $i=1,\ldots , 5$
such that $$ f =\sum_{i=1}^5c_iv_i^3.$$ The algorithm to find the $v_i$
is described in the proof.
\end{thm0}

\begin{proof}

In the pentahedral case we have $ f \in S^3\CC^4$.
Set $a=2, m=1$ in the Koszul flattening construction.
This corresponds to constructing $P_{ f }\colon \Hom(\CC^4,\bw 2\CC^{4})\to \Hom(\CC^4, \CC^4)$.
This is a $16\times 24$ matrix coming from the $4\times 6$ Koszul matrix
\[k_{2}=\bgroup\begin{pmatrix}
{-x_{1}}& {-x_{2}}& 0& {-x_{3}}& 0& 0\\
x_{0}& 0& {-x_{2}}& 0& {-x_{3}}& 0\\
0& x_{0}& x_{1}& 0& 0& {-x_{3}}\\
0& 0& 0& x_{0}& x_{1}& x_{2}
\end{pmatrix}\egroup,\]
and substituting the $4\times 4$ catalecticant matrix $ C_{ f_i}^1$ at each occurrence of $x_i$.
Using Macaulay2 \cite{M2} it can be checked that the kernel of $P_f$ has dimension $9$ and it is spanned by $9$ vectors in $\CC^{24}\simeq \Hom(\CC^4,\bw{2} \CC^{4})$ 
{written as $\sum w_{ij}x^i\wedge x^j$ that can be grouped as
$\{w_{01},w_{02},w_{12},w_{03},w_{13},w_{23}\}$ with $w_{ij} \in \CC^{4*}$ linear forms.} 

 The general element $M$ of the kernel can be computed, again with the help of Macaulay2,
by a random linear combination of the nine elements of the basis and it has exactly {\it five} 
eigenvectors $v_i$ (up to scale), in agreement with \propref{prop:wedge_eigen},
thus proving the existence and uniqueness statement of the theorem. Explicitly the five eigenvectors
(which are dual to the planes of the pentahedral) can be computed using the
$4\times 4$ minors of 
\[\bgroup\begin{pmatrix}x_{2}&       x_{3}&       0&       0&w_{01}\\
      {-x_{1}}&       0&       x_{3}&       0&w_{02}\\
      x_{0}&       0&       0&       x_{3}&w_{12}\\
      0&       {-x_{1}}&       {-x_{2}}&       0&w_{03}\\
      0&       x_{0}&       0&       {-x_{2}}&w_{13}\\
      0&       0&       x_{0}&       x_{1}&w_{23}\\
      \end{pmatrix}\egroup,\]
where the first $4$ columns express $k_3$.

The zero locus of these minors is indeed formed by the five points corresponding to $v_i$.
\end{proof}
\begin{remark}\label{rmk:BCMT}
The recent results in the nice paper \cite{BCMT} makes use of a different generalization of the classical methods for tensor decomposition.  Their methods make use of a (semi-)Hankel operator constructed from all possible catalecticant matrices, and compute the linear forms in the decomposition of a tensor utilizing a method of zero-dimensional root finding involving simultaneous eigenvectors of companion matrices.  In theory, this method eventually covers all cases to decompose a general symmetric tensor, however there is an essential difference with our methods.  That is, their method relies on numerical techniques in order to construct the Hankel operator. As a consequence, for example in the pentahedral case, their technique gives the ideal of the 5 points numerically, while in our algorithm, we have found the ideal of the 5 points symbolically, and numerical techniques are used only to compute the individual points.  
\end{remark}

\section{New methods for tensor decomposition: bundle method}\label{sec:bundle method}
\subsection{The bundle construction}\label{sec:bundle construction}

Let $L$ be the line bundle on $X$ which gives the embedding {$X\subset \PP (H^0(X,L)^{*}) = \PP W$.}  
In particular $L=\O(d)$ on $\PP^n = \PP V$ gives the embedding of the Veronese variety $X=v_d(\PP^n) = v_{d}(\PP V)$, where in this case 
{$ S^{d}V = W$,}
defined in the introduction. 

Let $E$ be a vector bundle on $X \subset \PP W$.
{In \cite{LO11}  a linear map $A_{f}$ was constructed, depending linearly on $ f \in W$, which comes from the natural contraction map}
\begin{equation}\label{Af construction}
H^0(E )\otimes H^0(E^*\otimes L)\rig{} H^0(L).
\end{equation}
From \eqref{Af construction} we get a linear map
{\[H^0(E )\otimes H^0(L)^{*}\rig{} H^0(E^*\otimes L)^{*},\]}
and this can be seen as a linear map
\begin{equation}\label{eq:afdef}
A_{f}\colon H^0(E )\rig{} H^0(E^*\otimes L)^{*}
\end{equation}
depending linearly on {$f\in H^{0}(L)^*$.}

Our starting point is the following result from \cite{LO11}.

\begin{prop0}\cite[Proposition~5.4.1]{LO11} \label{prop:baseloc}
Let $ f  = \sum_{i=1}^{r}v_{i} \in W$ with $z_{i}=[v_i] \in X \subset \PP W$  and put $Z = \{z_{1},\ldots,z_{r}\}$.
\begin{align*}
H^{0}(\I_{Z}\otimes E  ) &\subseteq \ker A_{ f } \\
H^{0}(\I_{Z}\otimes E  ^{*}\otimes L) &\subseteq \left(\im A_{ f }\right)^{\perp}.
\end{align*}
{The first inclusion is an equality if $H^{0}(E  ^{*}\otimes L)\rig{} H^{0}(E  ^{*}\otimes L_{\mid Z})$ is surjective. 
The second inclusion is an equality if $H^{0}(E  )\rig{} H^{0}(E  _{\mid Z})$ is surjective. } 
\end{prop0}

Recall that the {\it base locus} of a space of sections of a bundle is the common zero locus of all the sections of the space.
Again, let $E$ be a vector bundle on {$X\subset \PP (H^0(X,L)^*)$.} 
Let $f=\sum_{[v_i]\in Z}v_i\in  H^0(X,L)^*$.
Assume that $H^0(E  ^{*}\otimes L)\rig{} H^{0}(E  ^{*}\otimes L_{\mid Z})$ is surjective.
Then the kernel of $A_f\colon H^0(E  )\to H^0(E  ^{*}\otimes L)^{*}$
is equal $H^0(\I_Z\otimes E  )$. So, if the base locus of
$H^0(\I_Z\otimes E  )$ is $Z$ itself, then the decomposition of $f$ can be computed from the base locus of $\ker A_f$. 

Some advantages of our algorithm are now apparent.
First, $\ker A_{f}$ can be computed by an explicit matrix construction that we give below and methods from linear algebra can be used to compute this kernel. 
Recall that there is always a brute force method where one guesses a rank $r$, chooses $r$ linear forms $p_{i} = p_{i}(x_{0},\ldots,x_{n})$ each depending on $n+1$ parameters, and tries to solve the system of polynomials given by comparing coefficients on the expression $f = \sum_{i=1}^{r}p_{i}^{d}$.  

So if one compares our method to the brute force method, 
$\ker A_{f}$ consists of polynomials of lower degree than the original polynomial $f$, and in general a system of polynomials of lower degree should be easier to solve than one consisting of polynomials of higher degree. 
We tested our algorithm using $r$ randomly chosen linear forms $l_{i}$, and tried to decompose (the expanded form of) $f = \sum_{i=1}^{r}l_{i}^{d}$. The brute force method quickly fails (we run out of memory and time) even for the pentahedral example, where our algorithm succeeds in less than one second.

\subsection{A side remark on presentations}\label{sec:presentations} A \emph{presentation} of a bundle is what allows us to make the transition between vector bundles and matrices.  For a bit more on presentations we invite the reader to consult \cite[Chapter 6C]{Eisenbud_syzygies}.
A presentation of a bundle $E  $  on $\PP V$ is constructed as follows. 
We have the (finite) minimal resolution of $E$, which is
\begin{equation}\label{eq:cx}
\ldots\rig{}L_{2}\rig{}L_{1}\rig{}E\rig{}0,
\end{equation}
where each $L_{i}$ is a direct sum of line bundles,
and has the property that the induced map $H^{0}(L_{1})\rig{}H^{0}(E)$ is surjective.
This follows because the resolution is constructed as the sheafification of the corresponding resolution of graded modules.

Moreover we have the (finite) minimal resolution of $E^{*}$ which is
\begin{equation}\label{eq:dualcx}
\ldots\rig{}L^{*}_{-1}\rig{}L^{*}_{0}\rig{}E^{*}\rig{}0,
\end{equation}
where each $L_{i}$ is again a direct sum of line bundles,
which has the property that, even tensoring with a line bundle $L$,
the induced map $H^{0}(L^{*}_{0}\otimes L)\rig{}H^{0}(E^{*}\otimes L)$
is surjective.
Dualizing, we get a double resolution
\[\ldots \rig{}L_{2}\rig{}L_{1}\rig{p}L_{0} \rig{} L_{-1}\rig{}\ldots,\]
where $\im(p)=E$. {The map $L_{1}\rig{p}L_{0}$ gives the {\it presentation} of $E$.}
We get the composition
\begin{equation}\label{eq:Pfsecond}
P_{f} = \beta\circ A_{f}\circ \alpha \colon H^{0}(L_{1})\rig{\alpha}H^{0}(E)\rig{A_{f}}H^{0}(E^{*}\otimes L)^{*}
\rig{\beta}H^{0}(L^{*}_{0}\otimes L)^{*}
,\end{equation}
where $\alpha$ is surjective (because of \eqref{eq:cx}, {which is the sheafification of the minimal free resolution of the module
$\oplus_m H^0(E(m))$}) and $\beta$ is injective (because of \eqref{eq:dualcx}).
It can be shown that the matrix of $P_{f}$ can be constructed {just by the presentation $p$} with a block structure obtained as follows. If $p$ has a linear entry depending on $x_{i}$, substitute the catalecticant matrix $C_{\frac{\partial f}{\partial x_{i}}}$ for each $x_{i}$ (where the size of the matrix is to be determined by $L_{1}$ and $L_{0}$), and if $p$ has non-linear entries, replace each monomial in the $x_{i}$ by the catalecticant matrix of the associated derivative of $f$, for more details see \cite[Section~8.3]{LO11}.

Hence $\rank A_{f}=\rank P_{f}$ and, even more important
$\ker A_{f}$ and $\ker P_{f}$ have the same base locus.
The advantage is that $P_{f}$ can be explicitly computed
from a matrix with entries homogeneous polynomials,
and $\ker P_{f}$ is spanned by computable polynomials.

\begin{example}\label{ex:pfexample}
For a basic example, suppose $d=2m+1$, and let $E  =Q (m)$ be a twist of the quotient bundle on {$\PP^{2 }=\PP V$, as in the case of \thmref{thm:k2bound}} 

This gives a presentation
\[L_1=\O(m)\otimes V\rig{p}\O(m+1)\otimes V^{*}=L_0,\]
which is part of the Koszul complex
{
\[0\rig{}\O(m-1)\rig{}\O(m)\otimes V\rig{p}\O(m+1)\otimes V^{*}\rig{}\O(m+2)\rig{}0\]
}
After an appropriate choice of basis, $p$ may be represented by the matrix
\[\begin{pmatrix}
0&x_2&-x_1\\
-x_2&0&x_0\\
x_1&-x_0&0\\
\end{pmatrix},\]
which is one of the Koszul matrices we have already seen.
Then $H^0(L_1)=S^{m}V^{*}\otimes V$.

Now let $L=\O(2m+1)$. For any $ f \in S^{2m+1}V^{*}=W$, $A_ f $ is the morphism from
$H^0(Q (m))$ to its dual  $H^0(Q ^{*}(m+1))^{*}$ and 
{\[P_f \colon H^0(L_1)=\Hom(S^{m}V,V) \rig{} \Hom(V,S^{m}V)=H^0(L_0^*\otimes L)^*\]} is represented by the matrix
\[\begin{pmatrix}
0&C^{m}_{ f_{2}}&-C^{m}_{ f_{1}}\\
-C^{m}_{ f_{2}}&0&C^{m}_{ f_{0}}  \\ 
C^{m}_{ f_{1}}&-C^{m}_{ f_{0}}&0\\
\end{pmatrix},\]
where $C^{m}_{ f_{i}} \colon S^{m}V \to S^{m}V^{*}$ are catalecticant matrices of the partial derivatives $ f_{i}=\frac{\partial f }{\partial x_{i}}$. From this presentation, we can see also that $P_{f}$ is skew-symmetric.
Note that when $ f =x_0^{2m+1}$ is the power of a linear form then 
the above matrix has rank $2$, which indeed is the rank of $E  =Q (m)$.  We further remark that then the principal Pfaffians of this matrix give equations for secant varieties of $v_{2m+1}(\PP^{2})$, see \cite{LO11}.
\end{example}
\subsection{Vector bundles, statement and proofs of the main results}\label{sec:proofs}
For this section we assume the following general setup.
Let $X$ be an algebraic variety and 
let $L$ be the line bundle on $X$ which gives the embedding {$X\subset \PP (H^0(X,L)^*) = \PP W$.}
For $f \in W$, let {$f=\sum_{i=1}^k v_i$} be a minimal decomposition, {let $z_i=[v_i]\in\PP(W)$ 
be the corresponding points and let $Z=\{z_1,\ldots, z_r\}$}
For a vector bundle $E$ over $X$,  construct the map $A_f\colon H^0(E  )\to H^0(E  ^{*}\otimes L)^{*}$ as above.
\begin{prop0}\label{prop:equality}
Assume that 
$\rank A_f=k\cdot \rank E  $.
Then we have 
\begin{align*}
H^{0}(\I_{Z}\otimes E  ) &= \ker A_{ f } \\
H^{0}(\I_{Z}\otimes E  ^{*}\otimes L) &= \left(\im A_{ f }\right)^{\perp}.
\end{align*}
\end{prop0}
\begin{proof}\propref{prop:baseloc} says that we always have the inclusion $H^{0}(\I_{Z}
\otimes E  ) \subseteq \ker A_{ f } $.  But in fact we get equality because of the following dimension argument.
\[
\codim(H^{0}(I_{Z}\otimes E)) \leq k\cdot \rank E = \rank (A_{f}) = \codim(\ker A_{f})
.\]
The same argument applies to the second equality.
\end{proof}

We can give a general criterion
\begin{theorem}\label{thm:metabundle}
%
Assume that 
$\rank A_f=k\cdot \rank E$
and 
$$H^0(\I_Z\otimes E  )\otimes H^0(\I_Z\otimes E  ^{*}\otimes L)\to H^0(\I_Z^2\otimes L)$$
is surjective.

Assume that $X$ is not $k$-weakly defective, then the common base locus of $\ker A_f$ and of $\left(\im A_{f}\right)^{\perp}$ is given by $Z$ itself, hence $Z$ can be reconstructed by $f$.

\end{theorem}
Note that the notion of $k$-weakly defective has been introduced in \cite[Definition 1.2]{CC02}.
{See also \thmref{thm:uniqueness} and the paragraph thereafter.}

\begin{proof}
We can use the equalities of\propref{prop:equality}.
Assume that  the common base locus of $\ker A_f$ and of $\left(\im A_{f}\right)^{\perp}$ contains $Z\cup\{z'\}$, then every element of $H^0(\I_{Z^2}\otimes L)$ vanishes doubly on $z'$,
that is the general hyperplane section {(in the system $H^0(L)$)} of $X$
which is singular at $Z$ is singular also at $z'$, and this contradicts \cite[Theorem~1.4]{CC02}.
\end{proof}

\begin{example} The example of plane sextics $v_6(\PP^2)$,
where $Z$ is given by $8$ points is instructive. In this case,
set $E=\O(3)$ and $L=\O(6)$, it is a classical fact that the map
$H^0(\I_Z\otimes E  )\otimes H^0(\I_Z\otimes E  ^{*}\otimes L)\to H^0(\I_Z^2\otimes L)$
is NOT surjective. Indeed, it is known from \cite{LO11} that the catalecticant minors are not enough to give all the equations of $\sigma_8(v_6(\PP^2))$, a Koszul flattening is needed in addition.
In this case, the kernel of $A_f$ is spanned by a pencil of plane cubics and has base locus in $9$ points,
one more than the original eight. This is linked to the known classical fact that the pencil of plane cubics through $8$  points has an additional ninth base point.
Still, in this case, by (2) of \thmref{thm:IKbound}, Algorithm~\ref{catalg} succeeds to find the tensor decomposition.
\end{example}

The theorem applies in the presentation setting as well.
In the case $E  =\bw aQ (\delta)$, there is the presentation
$p\colon L_1=\bw aV\otimes\O(\delta)\to \bw {n-a}V^{*}\otimes\O(\delta+1)=L_0$
such that $\im p=E  $.
Then $H^0(L_1)=\Hom(S^{\delta}V,\bw aV)$
and an element of $H^0(L_1)$ goes to a section of $E  $ vanishing in $v$
if it corresponds to $M\in \Hom(S^{\delta}V,\bw aV)$
such that $M(v^{\delta})\in v\wedge\left(\bw {a-1}V\right)$, i.e. if $v$ is an eigenvector of $M$.
In this case, the base locus of the kernel can be studied directly in $H^0(L_1)$.
The criterion applies to specific tensors. 

We now come to the main practical result of this paper.  
\begin{algor0}[General algorithm to find tensor decomposition]\label{alg:metaalgorithm} 
\hfill\par
\textbf{Input:} $f \in S^{d}V$ where $\dim V = n+1$, $E$ is a convenient vector bundle on $\P{{}}(V)$, to be chosen.
\begin{enumerate}
\item Construct the map $A_{f}$ as defined in \eqref{eq:afdef}, where $L=\O(d)$.
\item Compute $\ker A_{f}$. If $\ker A_{f}$ is trivial, stop, this method fails.
\begin{enumerate}
\item note that $ \rank (f) \geq \frac{\rank (A_{f})}{\rank (E)} $.
\end{enumerate}
\item Find the base locus $Z'$ of $\ker A_{f}$ by explicitly computing $\ker P_{f}$ as in \eqref{eq:Pfsecond}.
\begin{enumerate}
\item If $Z'$ does not consist of a finite set of reduced points, stop; this method fails.
\item Otherwise continue with $Z'= \{[v_1],\ldots , [v_s]\}$.
\end{enumerate}
\item Solve the linear system defined by $f=\sum_{i=1}^sc_iv_i^d$ in the unknowns $c_i$.
\end{enumerate}
\par\textbf{Output:} If there is a unique solution to (4), this is the unique Waring decomposition of $f$.  Else we possibly find many minimal Waring decompositions of $f$.
\end{algor0}
Note that \thmref{thm:metabundle} applies to step (3) in that it says that the baselocus computed actually consists of the linear forms used in the construction of $f$.

\begin{remark}\label{rmk:failure1}
In practice one finds that the algorithm will fail if either $\ker P_{f}$ is trivial, in which case we must conclude that the rank of the input is too large to be decomposed by this method (see also~\rmkref{rmk:failure2} below), or the base locus of $\ker P_{f}$ contains infinitely many points, in which case we cannot determine the decomposition.  In the latter case we conclude that there is a positive dimensional variety on which our input tensor lies. In some pathological cases it may happen
that the base locus of $\ker P_{f}$ is non reduced, also in this case we cannot determine the decomposition.
\end{remark}

The following theorem has already been applied in many specific examples,
{(starting from the case $\textrm{rk} E=1$ which corresponds to the catalecticant case), } that show the versatility and power of the result.  

\begin{theorem}\label{thm:metathm}
Let $f\in S^{d}V$, and set $Z' = \baseloc (\ker A_{f})$, and let $\text{length}(Z')=s$.
Assume that  {$ \rank (f) = \frac{\rank (A_{f})}{\rank (E)} $}. 
Then for any minimal decomposition $f=\sum_{i=1}^r c_iv_i$ with $z_{i}=[v_{i}]\in X$,
 $c_{i}\in \CC$ and $Z = \{z_{1},\ldots,z_{r}\}$, we have $Z \subset Z'$.

If $\text{length}(Z') <\infty$, then Algorithm~\ref{alg:metaalgorithm} produces all minimal Waring decompositions of $f$, in particular if the solution is unique, we find the unique Waring decomposition of $f$.
\end{theorem}

\begin{proof}
Indeed let $f=\sum_{i=1}^r \mu_iv_i$ be a decomposition with minimal rank and set $Z = \{z_{1},\ldots,z_{r}\}$, with $z_{i}=[v_{i}]$.
By\propref{prop:equality} we have 
 $H^{0}(\I_{Z} \otimes E  ) = \ker A_{ f } $.  
It follows that $Z'$ is also the base locus of $H^{0}(I_{Z}\otimes E)$, hence $Z'\supseteq Z$. 
If $Z'$ is finite, then we can try the linear system in step (4), and with finitely many attempts we find a decomposition. 
 Uniqueness implies that the linear system given by $f=\sum_{i=1}^r c_iv_i$ has a unique solution, thus producing the unique minimal
 Waring decomposition.
\end{proof}
To further clarify, we restate the previous algorithm in a bit more detail in the case $E = \bw{a} Q(m)$, and in the presentation setting:

\begin{algor0}[Koszul Flattening Algorithm (General)]\label{koszulgeneral}
\hfill\par
\textbf{Input:} $ f  \in S^{d}V$, where $V$ has basis $\{x_{0},\ldots,x_{n}\}$.
\begin{enumerate}
\item{} Compute $\delta_{-} =  \left\lfloor\frac{d-1}{2}\right\rfloor$ and $\delta_{+} =  \left\lceil\frac{d-1}{2}\right\rceil$, and choose $a=\left\lceil \frac{n}{2}\right\rceil$, the Koszul flattening to use.
\item Construct the Koszul matrices $k_{p}$ for $p=n+1-a, n+2-a$
\item Construct the catalecticants $C_{ f_{i}}\colon S^{\delta_{+}}V \to S^{\delta_{-}}V$, of $ f_{i} = \frac{\partial f }{\partial x_{i}}$ for each $i$.
\item{} Construct the matrix $P_{ f } \colon \Hom(S^{\delta_{+}} V, \bw {a} V)\to \Hom(\bw {n-a} V, S^{\delta_{-}}V)$ by substituting $C_{ f_{i}}$ for $x_{i}$ in the matrix $k_{n+1-a}$.
\item{} Compute a basis $\{M_{1},\dots,M_{t}\}$ of $\ker P_{ f }$, and associate vectors of polynomials $\vec{w_{i}}$ to each $M_{i}$.  If $\ker P_{ f }$ is trivial, stop; this method fails.
\item{} Compute the eigenvectors $\{v_1,\ldots , v_s\}$ of a general element in $\ker P_{ f }$ as follows:
\begin{enumerate}
\item For each $\vec{w}_{i}$ compute the $\binom{n+1}{a-1} \times\binom{n+1}{a-1}$ minors of the block matrix 
$\left(\begin{array}{c|c}
k_{n+2-a}& \vec{w}
\end{array}\right)$ and store these minors in an ideal $J$.
\item Find the set $\{v_{1},\ldots,v_{s}\}$ of common eigenvectors of all $M$ in $\ker P_{ f }$ by computing the zero-set of $J$.  If $s$ is infinite, stop, the method fails. Otherwise continue.
\end{enumerate}
\item{} Solve the linear system $ f =\sum_{i=1}^sc_iv_i^d$ in the unknowns $c_i$.
\end{enumerate}
\par\textbf{Output:} The unique Waring decomposition of $ f $.
\end{algor0}

\begin{remark}\label{rmk:failure2}
Notice that in the middle of our algorithm (step (5)), the algorithm will fail if a certain matrix we construct has a trivial kernel.   This is an indication that the rank of the tensor is higher than the ranks which we can successfully decompose. We give precise bounds on the rank of the tensor (depending on the number of variables) for the success of this algorithm in ~\secref{sec:sufficient conditions}.  So the failure of the algorithm provides a lower bound for the rank of the input tensor.
\end{remark}

\subsection{Quintic and Pentahedral examples revisited}\label{sec:quintic pentahedral again}
Suppose $d=2m+1$.
In the quintic example we set $m=2$ and $E  =Q (2)$ on $\PP^2$.
The general section of $Q (2)$ vanishes on $7$ points.

In the pentahedral example we set $m=1$ and $E  =Q ^*(2)$ on $\PP^3$.
The general section of $Q ^*(2)$ vanishes on $5$ points.

An interesting remark is that asking that $h^0(Q (m))\ge nc_n(Q (m))$
(necessary condition to get that the zero locus of a section of $Q (m)$ is given by $c_n(Q (m))$\footnote{{We recall a few facts about the Chern class $c_{n}$ in ~\secref{sec:Chern}.}}
arbitrary points, {by counting parameters}), one checks that
the only solution for $m\ge 2$ to the diophantine equation resulting from\propref{prop:CS} is given by $n=m=2$ (the quintic example).

In the same way, asking that $h^0(Q ^*(m))\ge nc_n(Q ^*(m))$,
the only solution to the diophantine equation resulting from\propref{prop:wedge_eigen} for $m\ge 2$ is given by $n=3, m=2$ (pentahedral example).
In particular, our methods find all of the already known cases where the decomposition of a general symmetric tensor is unique, but do not provide any more uniqueness than was already known classically. This gives evidence that there may not be any other cases where uniqueness holds for the general tensor.

\section{Rank bounds for feasibility of our algorithms}\label{sec:sufficient conditions}
In this section we give bounds, depending on the number of variables and the rank of the tensor, for when our algorithm will succeed to produce the Waring decomposition of a given symmetric tensor.

Below in \thmref{thm:applicabilityP2} and \thmref{thm:applicabilityP3} we study bundles (over $\PP^{2}$ and $\PP^{3}$ respectively) twisted by an arbitrary integer $m$.  More generally we consider $\PP^{n}$ and apply these results to prove \thmref{thm:k2bound} and \thmref{thm:knbound} in the specific case that $d=2m+1$.

\begin{theorem}\label{thm:applicabilityP2}
Let $Q$ be the quotient bundle on $\PP^2$ and 
let $Z$ be a finite collection of $s$ general points in $\PP^2$.
\begin{enumerate}
\item[(i)] If $s\le \frac{1}{2}(m+3)(m+1)$ and $m\geq 0$, then $h^0(Q \otimes \I_Z(m))=(m+3)(m+1)-2s$;

this is equivalent to the fact that $H^0(Q (m))\rig{}H^0(Q _Z(m))$ is surjective.

\item[(ii)] If $s\le\frac{1}{2}(m^2+3m+4)$ ($m\ge 2)$ (or $s\le 3$ for $m=1$), then {the} 
base locus
of $H^0(Q \otimes \I_Z(m))$ is given by $Z$ itself.
\end{enumerate}
\end{theorem}
\begin{proof}
First note that, since $Z$ is {a general collection of points}, it is reduced. By semi-continuity, both statements can be proved  for a special 
collection $Z$. We prove the theorem by induction from $m-2$ to $m$,
so we have to distinguish the even and odd cases.
The starting cases {$m=0$} and {$m=1$} can be easily checked directly.
If $m=1$ then (i) says: if $s\le 4$, then $h^0(Q \otimes \I_Z(1))=8-2s$;
{Indeed  an element $A\in H^0(Q(1))$ corresponds to a traceless endomorphism of $V$
and by \lemref{lem:quotient} it vanishes exactly at the eigenvectors of $A$.
The space $E(v_1,\ldots v_s)$ of traceless endomorphisms which have $s$ general vectors $v_1,\ldots v_s$ as eigenvectors has dimension $8-2s$ for $s\le 4$, which proves the initial case
$m=1$ of (i). Moreover the common eigenvectors of the endomorphisms in  $E(v_1,\ldots v_s)$ for $s\le 3$
are  $v_1,\ldots v_s$ themselves, which proves the initial case
$m=1$ of (ii). The initial case $m=0$ of (i) is easier, while the initial case $m=2$ of (ii)
corresponds to the plane quintic example developed in the introduction.}
In this case $Z$ is given by seven points, $H^0(Q \otimes \I_Z(2))$ is one dimensional and
it is spanned by the unique section of $Q(2)$ vanishing on $Z$, here it plays an essential role
that $c_2(Q(2))=7$. This example may be checked also in {Macaulay2}.


%
%

First we prove (i). {Let $C\simeq\PP^1$ be a smooth conic. We recall that we have the splitting
$Q (m)_{|C}=\O_{|\PP^1}(2m+1)\oplus \O_{|\PP^1}(2m+1)$. Indeed, from all the possible splittings 
$\O_{|\PP^1}(\alpha)\oplus \O_{|\PP^1}(\beta)$ with $\alpha+\beta=4m+2$ the most balanced one is the only compatible
with the vanishing $H^0(Q(-1)_{|C})=0$ which follows from the exact sequence
\[0\rig{}\O(-3)\rig{}Q(-1)\rig{}Q(-1)_{|C}\rig{}0.\]}

We choose to specialize $\tilde s=\min(s,2m+2)$
points as general points
on {$C$}, and we call $Z'\subset\PP^1$ the subcollection obtained.
Let $Z''$ be given by the remaining points. Note that 
$\frac{1}{2}(m^2+4m+3)-\frac{1}{2}(m-2)^2+4(m-2)+3=2m+2$, so that the length of $Z''$ satisfies the
assumption of the inductive {step.}
There is a natural exact diagram
$$\begin{array}{ccccccccc}
&&0&&0&&0\\
&&\dow{}&&\dow{}&&\dow{}\\
0&\rig{}&\I_{Z''}(-2)&\rig{}&\I_{Z}&\rig{}&\I_{Z',C}&\rig{}&0\\
&&\dow{}&&\dow{}&&\dow{}\\
0&\rig{}&\O(-2)&\rig{}&\O&\rig{}&\O_C&\rig{}&0\\
&&\dow{}&&\dow{}&&\dow{}\\
0&\rig{}&\O_{Z''}&\rig{}&\O_{Z}&\rig{}&\O_{Z'}&\rig{}&0\\
&&\dow{}&&\dow{}&&\dow{}\\
&&0&&0&&0\\
\end{array}$$

Tensoring the first row by $Q(m)$ we get
 the exact sequence
\begin{equation}\label{eq:seqsplit}
0\rig{}Q (m-2)\otimes \I_{Z''}\rig{}Q (m)\otimes \I_Z
\rig{}\I_{Z'|\PP^1}(2m+1)\oplus \I_{Z'|\PP^1}(2m+1)\rig{}0
,\end{equation}
and the associated cohomology sequence
\[0\rig{}H^0(Q (m-2)\otimes \I_{Z''})\rig{}H^0(Q (m)\otimes \I_Z) 
\rig{}\CC^{4m+4-2\tilde s}.\]
Hence
\[h^0(Q (m)\otimes \I_Z)\le h^0(Q (m-2)\otimes \I_{Z''})+4m+4-2\tilde s,\]
and by the inductive assumption we get the inequality
\[h^0(Q \otimes \I_Z(m))\le m^2+4m+3-2s.\]
By considering the sequence
\[0\rig{}Q \otimes \I_Z(m)\rig{}Q (m)\rig{}\C^2\otimes\O_Z\rig{}0,\]
we get the opposite inequality, then the result. 

The proof of (ii) follows the same inductive step and we specialize the points in a similar manner.
We choose to specialize now only $\min(s,2m+1)$
points as general points
on a smooth conic $C\simeq\PP^1$. {An anonymous referee pointed out that the case $m=3$, $s=11$ cannot be proved by induction
because the case $m=1$ our statement holds only for $s\le 3$. { Indeed the case $m=3$, $s=11$
has to be proved separately (by direct computation). Our computation in Macaulay2 that accomplishes this can be checked and repeated using the file ``General Kappa method.m2'' (see  ~\secref{sec:M2})}
and setting the parameters $s=11$, $d=7$, $n=3$. } 

For any $z'\notin Z$
we have to prove that there is a section $\sigma\in H^0(Q \otimes \I_Z(m))$
such that $\sigma(z')\neq 0$. Note that {by the inductive assumption} there is $\sigma'\in H^0(Q \otimes \I_{Z''}(m-2))$
such that $\sigma'(z')\neq 0$. Let $f$ be the equation of $C$. If $z'\notin C$ then $\sigma=\sigma'f$ works.
The proof of (i) {(specifically the sequence \eqref{eq:seqsplit} specialized to fewer points)} shows that 
$$H^0(Q \otimes \I_{Z}(m))\rig{}H^0(Q \otimes \I_{Z',C}(m))$$
is surjective.
If $z'\in C$ there is a section $\sigma''\in H^0(Q \otimes \I_{Z',C}(m))$ which does not vanish in $z'$
(here we need that $\text{length }Z'\le 2m+1$) and there is 
a section in $H^0(Q \otimes \I_{Z}(m))$ which restricts to $\sigma''$ and then does not vanish in $z'$, as we wanted.
\end{proof}

\begin{remark}
Note that when $m$ is even ($m\ge 2$) and $2s=m^2+4m+2$ then $H^0(Q \otimes \I_{Z}(m))$ is generated by one section,
and the number of points in the base locus is given by $c_2(Q (m))=1+m+m^2\ge\frac{1}{2}(m^2+4m+2)$ and the equality holds only for $m=2$.
The case $m=2$ is indeed special, and it is the case of $S^5\CC^3$ addressed in the introduction.
The corresponding tensor decomposition of $S^5\C^3$ (considered here in the introduction) was noted in \cite{LO11}.\end{remark}

The following theorem is a generalization of \thmref{thm:applicabilityP2} from $\PP^2$ to $\PP^3$.  
In principle it is possible to obtain similar theoretical bounds for $\bw aQ $ on $\PP^n$, but it does not seem easy to find sharp bounds
like those in \thmref{thm:applicabilityP2} on $\PP^2$. See the \rmkref{rmk:rationalquartic} which shows that this
topic can be tricky.  We will continue to use $Q$ for the quotient bundle on any $\PP^{3}$ (and later we will use the same $Q$ for the quotient bundle on $\PP^{n}$), and when we restrict this quotient bundle to a smaller space we will indicate that as $Q_{|\PP^{2}}$ for example -- we hope the usage will be clear in context. Recall, from the Euler sequence, that $h^0(\PP^3,Q(m))=\frac 12(m+4)(m+2)(m+1)$ and $h^0(\PP^3,Q^*(m+1))=\frac 12(m+4)(m+3)(m+1)$ for $m\ge 0$.

In the proof of \thmref{thm:k2bound} there was extra symmetry (that the bundles $E$ and $E^{*}\otimes L$ happened to be isomorphic in the case $E=Q(m)$ and $L =\O(2m+1)$). Without this symmetry, we will need two analogous pairs of statements (stated next) in order to apply \propref{prop:baseloc} for the proof of \thmref{thm:knbound}.

\begin{theorem}\label{thm:applicabilityP3}
Let $Q $ be the quotient bundle on $\PP^3$ and 
let $Z$ be a finite collection of $s$ general points in $\PP^3$.\begin{enumerate}
\item[(i)] If 
$s\le \left\lfloor\frac{1}{3}(h^0(Q (m))-m+2)\right\rfloor
= \left\lfloor\frac{1}{3}\left(\frac 12(m+4)(m+2)(m+1)-m+2\right)\right\rfloor$ and $m\geq 0$,
then
$h^0(Q \otimes \I_Z(m))=\frac 12(m+4)(m+2)(m+1)-3s$.
This is equivalent to the fact that $H^0(Q (m))\rig{}H^0(Q _Z(m))$ is surjective.

\item[(ii)] If $s\le\frac{1}{3}\left(\frac 12(m+4)(m+2)(m+1)-\frac{m^2}{2}-\frac{3m}{2}+5\right)$ ($m\ge 2)$
 (or $s\le 4 $ for $m=1$), then the base locus
of $H^0(Q \otimes \I_Z(m))$ is given by $Z$ itself.

\item[(iii)] If 
$s\le \left\lfloor\frac{1}{3}(h^0(Q^* (m+1))-m+1)\right\rfloor
= \left\lfloor\frac{1}{3}\left(\frac 12(m+4)(m+3)(m+1)-m-2\right)\right\rfloor$ and $m\geq 0$,
then
$h^0(Q^* \otimes \I_Z(m+1))=\frac 12(m+4)(m+3)(m+1)-3s$.
This is equivalent to the fact that $H^0(Q^* (m+1))\rig{}H^0(Q^* _Z(m+1))$ is surjective.

\item[(iv)] If $s\le\frac{1}{3}\left(\frac 12(m+4)(m+3)(m+1)-\frac{m^2}{2}-\frac{m}{2}-8\right)$ ($m\ge 2$)
 (or $s\le 5$ for $m=1$), then the base locus
of $H^0(Q^* \otimes \I_Z(m+1))$ is given by $Z$ itself.
\end{enumerate}
\end{theorem}
\begin{proof}
By semi-continuity, each statement can be proved  for a special 
collection $Z$. We prove the theorem by induction from $m-1$ to $m$.
The starting cases $m=1$ can be easily checked directly.
Note that $Q (m)_{|\PP^2}=Q_{\PP^2}(m)\oplus\O_{\PP^2}(m)$.
We choose to specialize $\tilde s=\min\left(s,\left\lfloor\frac{1}{3}(h^0(Q (m)_{|\PP^2})\right\rfloor\right)$
points as general points on a hyperplane $\PP^2$
and we call $Z'\subset\PP^2$ the subcollection obtained.
Let $Z''$ be given by the remaining points. Let $g(m)$ and $f(m)$ respectively denote the numbers
\begin{equation}g(m)=\left\lfloor\frac{1}{3}(h^0(Q (m)_{|\PP^2})\right\rfloor\end{equation} and
\begin{equation}f(m)=\left\lfloor\frac{1}{3}\left(\frac 12(m+4)(m+2)(m+1)-m+2\right)\right\rfloor.\end{equation}
It is straightforward to check that $f(m)-f(m-1)=g(m)$,  so that the length of $Z''$ satisfies the
assumption of the inductive step.

Then we have the exact sequence
\[0\rig{}Q (m-1)\otimes \I_{Z''}\rig{}Q (m)\otimes \I_Z\rig{}\I_{Z'|\PP^2}\otimes Q _{|\PP^2}(m)\rig{}0,\]
and the associated cohomology sequence
\[0\rig{}H^0(Q (m-1)\otimes \I_{Z''})\rig{}H^0(Q (m)\otimes \I_Z)\rig{}\CC^{h^0(Q (m)_{|\PP^2})-3s}.\]  
Hence
\[h^0(Q (m)\otimes \I_Z)\le h^0(Q (m-1)\otimes \I_{Z''})+h^0(Q (m)_{|\PP^2})-3s,\]
and by the inductive assumption we get the inequality
\[h^0(Q \otimes \I_Z(m))\le (m+4)(m+2)(m+1)-3s.\]
By considering the sequence
$$0\rig{}Q \otimes \I_Z(m)\rig{}Q (m)\rig{}\C^3\otimes\O_Z\rig{}0,$$
we get the opposite inequality, then the result. 

The proof of (ii) follows the same inductive step and we specialize the points in a similar manner.
We choose to specialize now only $\min(s,\left\lfloor\frac{1}{3}\left((m^2+3m+4)+{{m+2}\choose 2}\right)\right\rfloor)$
points as general points
on $\PP^2$.
Now let $g'(m)$ and $f'(m)$ respectively denote the numbers
\begin{equation}g'(m)=\left\lfloor\frac{1}{3}\left((m^2+3m+4)+{{m+2}\choose 2}\right)\right\rfloor\end{equation}
 and
\begin{equation}f'(m)=\frac{1}{3}\left(\frac 12(m+4)(m+2)(m+1)-\frac{m^2}{2}-\frac{3m}{2}+5\right).\end{equation}
Also in this case it is possible to check that $f'(m)-f'(m-1)=g'(m)$. 

For any $z'\notin Z$
we have to prove that there is a section $s\in H^0(Q \otimes \I_Z(m))$
such that $s(z')\neq 0$. Note that there is $s'\in H^0(Q \otimes \I_{Z''}(m-1))$
such that $s'(z')\neq 0$. Let $h$ be the equation of $\PP^2$. If $z'\notin C$ then $s=s'h$ works.
The proof of (i) shows that 
$$H^0(Q \otimes \I_{Z}(m))\rig{}H^0(Q \otimes \I_{Z',\PP^2}(m))$$
is surjective.
If $z'\in \PP^2$ there is a section $s''\in H^0(Q \otimes \I_{Z',\PP^2}(m))$ which does not vanish in $z'$ (here we need (ii) of \thmref{thm:applicabilityP2}) and there is a section in $H^0(Q \otimes \I_{Z}(m))$ which restricts to $s''$ and then does not vanish in $z'$, as we wanted.
{The proof of (iii) is very similar to the proof of (i) because $Q^*(m+1)$ restricts on every plane to 
$Q_{\PP^2}(m)\oplus\O_{\PP^2}(m+1)$ . We set }
\begin{equation}g(m)=\left\lfloor\frac{1}{3}(h^0(Q^* (m+1)_{|\PP^2})\right\rfloor\end{equation} and
\begin{equation}f(m)=\left\lfloor\frac{1}{3}\left(\frac 12(m+4)(m+3)(m+1)-m-2\right)\right\rfloor.\end{equation}
We check that $f(m)-f(m-1)=g(m)$ and the initial case holds:
$f(0)=1$. Then we proceed exactly like in (i).

The proof of (iv) is very similar to the proof of (ii). So let 
\begin{equation}g'(m)=\left\lfloor\frac{1}{3}\left((m^2+3m+4)+{{m+3}\choose 2}\right)\right\rfloor\end{equation}
and let $f'(m) =  \lfloor \frac{1}{3}\left(\frac 12(m+4)(m+3)(m+1)-\frac{m^2}{2}-\frac{m}{2}-8\right)\rfloor $. 
Then one checks that $f'(m)-f'(m-1)=g'(m)$ and the initial case holds: $f'(2)=11$.

\end{proof}

The following theorem generalizes \thmref{thm:k2bound} to the case $n\ge 3$. Again, when $d$ is even, use  Algorithm~\ref{catalg}.
{In the following theorem we set $E=\wedge^{n-1}Q(m)=Q^*(m+1)$. Note that the 
pentahedral example of \ref{sec:quintic pentahedral again} corresponds to $n=3$ and $m=1$.
So we have the maps $A_f\colon H^0(Q^*(m+1))\to H^0(Q(m))^*$
and $P_f\colon \Hom(S^mV,\wedge^{n-1}V)\to \Hom(V,S^mV)$.
}

\begin{theorem}\label{thm:knbound} Suppose $n\ge 3$ and set $d=2m+1$. Let $f=\sum_{i=1}^rv_i^{d}$ be a form in $S^{d}V$ of rank $r$
{such that $f$ is general among the forms in $S^{d}V$ of rank $r$,} 
{let $z_i=[v_i]\in\PP(V)$ 
be the corresponding points and let $Z=\{z_1,\ldots, z_r\}$.}

\begin{enumerate}
\item {Let $Z'$ be the set of common eigenvectors (up to scalars) of $\ker P_f$.} If $n$ is even and $r\le {{m+n}\choose n}$, 
 then $Z'$ agrees with $Z$.
Moreover Algorithm~\ref{koszulgeneral} produces the unique Waring decomposition of $f$.
\item If $n$ is odd and $r\le {{m+n}\choose n}$ {let $Z'$ be the set of common eigenvectors (up to scalars)} of $\ker P_f$
and  $\left(\im P_{f}\right)^{\perp}$. Then $Z'=Z$ and Algorithm~\ref{koszulgeneral}, with this modification, produces the unique Waring decomposition of $f$.
\item {If $n=3$ and $r \leq \frac{1}{3}\left(\frac{1}{2}(m+4)(m+3)(m+1)-\frac{m^2}{2}-\frac{m}{2}-8\right)$ let $Z'$ be the set of common eigenvectors (up to scalars) of $\ker P_f$.} Then Algorithm~\ref{koszulgeneral} (with {$a=2$}) produces the unique Waring decomposition of $f$. 
\end{enumerate}
\end{theorem}
\begin{proof}
First notice that by \thmref{thm:uniqueness}, we know that in these cases we have a unique decomposition.  

Now we prove (1). By \cite[\S 7]{LO11} and \cite[Theorem~1.2.3]{LO11} we have that the natural map
\begin{equation}\label{fromLO11}
H^0(\bw{a}Q\otimes\I_Z(m))\otimes H^0(\bw{n-a}Q\otimes\I_Z(m))\rig{}H^0(\I_{Z^2}(2m+1))
\end{equation}
{is surjective.} {Moreover, 
the dimension of the image of the map in (\ref{fromLO11}) is
equal to the rank of the normal bundle
 of the variety {cut out} by the minors of size ${n\choose a}r+1$ of $P_f$ (where $f$ is considered here as a polynomial with 
variable entries).
 By the Alexander-Hirschowitz Theorem \ref{thm:AH} the 
space $H^0(\I_{Z^2}(2m+1))$ has the expected codimension $r(n+1)$ in the space $H^0(\O(2m+1))$,
hence the rank of the normal bundle is the expected one ${{n+2m+1}\choose n}-r(n+1)$
and it follows that the dimension of the scheme cut out by the minors is the expected one $r(n+1)-1$ at $[f]$.
In particular this scheme has a reduced irreducible component containing $[f]$
which is the $r$-secant variety to the $d$-Veronese embedding of $\PP^n$.
It follows that the scheme cut out by minors is smooth of the expected dimension at $f$, hence
we have the equality $\textrm{rk}(P_f)={n\choose a}r=\textrm{rk}(\bw{a}Q)\cdot \textrm{rk}(f)$,
otherwise, if $\textrm{rk}(P_f)$ is smaller, then the variety cut out by minors should have been
singular at $[f]$.
 Note that in \cite[Theorem~1.2.3]{LO11} it was fixed the value
$a=\lfloor{n/2}\rfloor$, but we may apply it as well
in the case $a=\lceil{n/2}\rceil$ because we get just the transpose map. }
If $n$ is even, then the symmetry of $P_f$ guarantees that $\ker P_f=\left(\im P_{f}\right)^{\perp}$. By \thmref{thm:metabundle}, with $E=\bw{a}Q$ and $L=\O(2m+1)$, we get that $Z=Z'$ {(the assumptions of \thmref{thm:metabundle} are satisfied by \thmref{thm:uniqueness}).}
The proof of (2) is analogous.

The proof of (3) follows the same lines of the proof of \thmref{thm:k2bound},
but we use \thmref{thm:applicabilityP3}
at the place of \thmref{thm:applicabilityP2}.
\end{proof}

\begin{remark}\label{rmk:rationalquartic} An interesting case is $n=4$, $d=3$, $r=7$, a defective case addressed in \cite{Ott2}. Set $V=\CC^5$, pick a general $f=\sum_{i=1}^7v_i^3$ 
with $Z=\{[v_1],\ldots ,[v_7]\}$
and construct $P_f\colon \Hom(V,\bw2 V)\to \Hom(\bw2 V, V)$. The locus $Z'$ of common eigenvectors
of $\ker P_f$ is the unique rational quartic curve in $\PP{{}}V$ passing through $Z$.
Note that a general element in $Hom(V,\bw2 V)$ has no eigenvectors, in agreement with \thmref{thm:counting}.
In more geometric terms, this means that all sections of $\bw2Q(1)$ vanishing on $Z$  also vanish on $Z'$.
This follows easily by the construction and by Terracini Lemma.
So, as a byproduct of Algorithm~\ref{koszulgeneral}, we have found an algorithm to write down the unique rational quartic
curve through $7$ general points in $\PP^4$. According to Ranestad-Schreyer (that we quote) the uniqueness is by Castelenuovo, (see the proof of Prop. 5.2 in their paper).
\end{remark}
\section{Using Chern classes to count the number of eigenvectors of a general tensor}\label{sec:Chern}

To a vector bundle $E$ on an algebraic variety $X$ are associated its Chern classes $c_i(E  )\in \mathcal{A}^i(X)$, where $\mathcal{A}(X)$ is the Chow ring of $X$.  The reader unfamiliar with Chern classes may wish to consult \cite{Hartshorne_tome} or \cite{OSS}.
For vector bundles on $\PP V$, we have
$\mathcal{A}^i(\PP V)=\Z$ and the Chern classes can be considered
as integers. 

The basic principle that we will use is the following.
If a vector bundle $E$ of rank $r$ on a variety $X$ has a section
vanishing on $Z$, and the codimension of $Z$ is equal to $r$,
then the class of $[Z]\in \mathcal{A}^r(X)$ 
is computed by $[Z]=c_r(E  )$.

The following proposition is a particular case of a more general result,
proved recently by Cartwright and Sturmfels with toric techniques \cite{CartwrightSturmfels2011}, who proved a conjecture stated in \cite{MR2296920},
where\propref{prop:CS} was proved in the case $m$ odd.

\begin{prop0}\label{prop:CS}(Cartwright-Sturmfels)
The number of eigenvectors (counted with multiplicity)
of a general  $M\in\Hom (S^mV,V)$ is equal to 
$$\frac{m^{n+1}-1}{m-1}.$$
\end{prop0}
\begin{proof}
{ By \lemref{lem:quotient}, } the number of eigenvectors of a general $M\in\Hom (S^mV,V)$ is equal to $c_n(Q(m))$.
In order to compute this number we use the formula {{(see \cite{OSS} section 1.2)}}
\[c_n(E  \otimes L)=\sum_{i=0}^n{{r-i}\choose{n-i}}c_i(E  )c_1(L)^{n-i}\]
for a vector bundle $E$ and a line bundle $L$.

In our case we have
\[c_n(Q(m))=\sum_{i=0}^nc_i(Q)c_1(\O(m))^{n-i}=
\sum_{i=0}^nm^{n-i}=\frac{1-m^{n+1}}{1-m},\]
where we have used the well known fact that $c_i(Q)=1$, which follows immediately from the exact sequence
\[0\rig{}\O(-1)\rig{}\O\otimes V\rig{}Q\rig{}0.\qedhere\]
\end{proof}

\begin{prop0}\label{prop:wedge_eigen}
 The number of eigenvectors (counted with multiplicity) of a general  $M\in\Hom (S^mV,\bw {n-1}V)$ is equal to 
\[
\frac{(m+1)^{n+1}+(-1)^n}{m+2}
.\]
\end{prop0}
\begin{proof}
{ By  \lemref{lem:quotient}, }the number of eigenvectors of a general $M\in\Hom (S^mV,\bw {n-1}V)$ is equal to $c_{n}(\bw{n-1}Q(m))=c_n(Q ^*(m+1))$, which can be computed in a similar way by noting that $c_{i}(Q^{*}) = (-1)^{i}$. To follow the previous proof, we set $m'= m+1$.
\begin{multline*}
c_n(Q^{*}(m'))=\sum_{i=0}^nc_i(Q^{*})c_1(\O(m'))^{n-i}
=
\sum_{i=0}^n (-1)^{i} (m')^{n-i}\\
= 
(-1)^{n}\sum_{i=0}^n (- m')^{n-i}=
(-1)^n\frac{1-(-m')^{n+1}}{1+m'}=
\frac{(m')^{n+1}+(-1)^n}{m'+1}.
\qedhere\end{multline*}
\end{proof}

\section{Macaulay 2 Implementation}\label{sec:M2}
The tensor decomposition algorithms in this article could be implemented in a variety of computational algebra or computational linear algebra packages.  We chose to implement our algorithms in Macaulay 2 because we found many of the procedures we would need were already implemented. Our algorithms may be easily adapted to other languages,  depending on the desired features.  Our code may be found online with the ancillary materials accompanying the arXiv version of our paper or by contacting either author.  {In this section we have tried to use the verbatim text style to indicate Macaulay2 input.}

\subsection{The catalecticant algorithm implementation}\label{sec:cat method}
Our first example is an implementation of the catalecticant algorithm (Algorithm~\ref{catalg}).  For this example we work with degree $d$ polynomials on $n+1$ variables.  This is the file ``cat\_method.m2''.  Our experiments show that working over a prime characteristic base field, or also the rational numbers, for relatively small $d$ and $n$, the catalecticant method quickly computes the decomposition within the range of \thmref{thm:IKbound}, and sometimes when the degree is low, even succeeds for slightly higher ranks than predicted by the bound.
For this implementation the user can change the values of degree \verb#d# and projective dimension \verb#n# as well as the rank \verb#s# of the test polynomial at the beginning of the file. The ring \verb#R# can be either over the rationals or over a prime characteristic field.

A polynomial \verb#ff# over ground field \verb#KK = ZZ/p# or \verb#QQ# is constructed as the sum of \verb#s# \verb#d#$^{th}$ powers by the following.
\begin{verbatim}
R  = KK[x_0..x_n]
ff=sum(s,i->(random(1,R))^d)
\end{verbatim}

Next, we construct a map that computes the ``most square'' catalecticant matrix associated to a given input polynomial. We can do this conveniently by using ceiling and floor commands to define the degrees, the basis command to define vectors of appropriate sizes for what would be the labels of the rows and columns of the matrix and the \verb#diff# command to construct the matrix.
\begin{verbatim}
af = floor(d/2)
ac = ceiling(d/2)
xaf = basis(af,R)
xac = basis(ac,R)
catalecticant=f->diff(transpose xac,diff(xaf,f)) 
\end{verbatim}

The next step is to find the base locus of the kernel of the catalecticant matrix.  
First we compute the generators of the kernel, then we convert these integer vectors to an ideal of polynomials using the basis of the base space, and finally we decompose the radical of the ideal by first computing the saturation (this sometimes results in a speed-up, and is justified because the saturation has the same {scheme}-theoretic structure, and this is all that concerns us with this application). 
\begin{verbatim}
K =gens kernel catalecticant ff
I = ideal(xaf*K)
L = decompose saturate I
\end{verbatim}
The list \verb#L# contains the apolars of the linear forms which will (up to scale) be used to write the decomposition.

Next we construct a ring \verb#S# by appending constants \verb#c_i# whose number is the number of points in the base locus of the kernel. 
\begin{verbatim}
S  = KK[x_0..x_n,c_0..c_(length L -1)]
\end{verbatim}
 We construct the polynomial \verb#fc# $= \sum_{v_{i}\in \hbox{\verb#L#}} c_{i}v_{i}^{*}$, where the $v_{i}^{*}$ are the apolar forms to the $v_{i}$.  We accomplish the swap between a form its apolar form within the summation as follows.
\begin{verbatim}
bS = sub(basis(1,R),S)
fc = sum(length L,i-> 
c_i*((bS*(mingens kernel diff(bS, transpose mingens sub(L_i,S))))_(0,0))^d );
\end{verbatim}
 
Next we solve the linear system on the \verb#c_i# obtained by setting equal to zero all \verb#d#$^{th}$ derivatives of the expression \verb#fc-ff#. 
\begin{verbatim}
Ic= ideal(sub(diff(basis(d,R),ff),S) - diff(sub(basis(d,R),S),fc));
Vc = decompose saturate Ic
\end{verbatim}
 Finally we substitute the found values for the \verb#c_i# into the polynomial \verb#fc# check to see if the decomposition succeeded. 
\begin{verbatim}
FF =sub(substitute(fc,S/Vc_0),R)
FF-ff
\end{verbatim}

Our tests succeeded to find decompositions for the following  \verb#n#, \verb#d#, and \verb#s# for example.
\begin{verbatim}
-- n=2: (d=3, s=1),  (d=4, s<=4),  (d=5, s<=4),  (d=6, s<=8) 
-- n=3: (d=3, s=1),  (d=4, s<=7),  (d=5, s<=7),  (d=6, s<=16)
-- n=4: (d=3, s=1),  (d=4, s<=10), (d=5, s<=10), (d=6, s<=16).
\end{verbatim}

\subsection{The Koszul flattening algorithm implementation}
Our second example is the implementation of the Koszul flattening algorithm (Algorithm~\ref{koszulgeneral}). This is contained in the file ``General Kappa Method.m2''.

As before we tested our algorithm by taking a sum of a fixed number $s$ of powers of random linear forms, expanding the resulting polynomial, and then testing to see if our algorithm gave the correct decomposition.  Here we will describe the aspects of this algorithm that differ from the catalecticant algorithm.

As before, we construct a map that computes the ``most square'' catalecticant matrix associated to a given input polynomial of degree \verb#d-1#.  The degree drops because we will eventually feed this map the first partial derivatives of our input polynomial \verb#ff#.
Then we construct the Koszul complex.
\begin{verbatim}
M = ideal(basis(1,R))
RM = resolution M
\end{verbatim}

Using the \verb#diff# command again, we take a matrix from the Koszul complex and construct a block matrix, replacing each entry in the Koszul matrix with the catalecticant of the partial derivative of our test function with respect to the entry in the Koszul matrix.  This matrix is our Koszul flattening, where \verb#ka# indicates which map in the Koszul complex we are using. In this case 
\verb#ka=n+1-ceiling(n/2)#. The matrix \verb#K# corresponds to the map called $P_f$ in this paper.
\begin{verbatim}
K = diff(transpose RM.dd_ka, catalecticant ff)
\end{verbatim}

The base locus of the kernel of the Koszul Flattening \verb#K# is a set of (generalized) eigenvectors.  Later we will construct this base locus, for now we compute the generators of the kernel of the Koszul Flattening. 
\begin{verbatim}
KM = generators kernel K 
\end{verbatim}

The kernel of the Koszul flattening should be a vector space of polynomials, but at present it is expressed as integer vectors. The Koszul flattening \verb#K# is an \verb#a#$\times$\verb#a# blocked matrix 
(in the case $n$ even, with small variations in the odd case,
when the matrix is no longer square) of \verb#m#$\times$\verb#m# blocks where \verb#a = binomial(n+1,ka)# and \verb#m=binomial(n+d-ac-1,n)#. The kernel of \verb#K# respects this block structure.  So we convert each blocked integer vector in the kernel into a smaller vector where each of the \verb#a# blocks becomes a polynomial written in the basis of monomials previously defined at \verb#xaf#. Macaulay 2 can do these computations simultaneously on matrices and not just individual vectors. 
\begin{verbatim}
for i from 0 to a-1 do { G_i =submatrix(KM,{i*m..(i+1)*m-1},);}  
pG =xaf* G_0; for i from 1 to a-1 do { pG = pG||(xaf*G_i);}  
\end{verbatim}
The outcome \verb#pG# is a matrix, each row of which is a vector of polynomials in the kernel of \verb#K#.

The kernel of \verb#K# is a subspace of $\Hom(S^{\hbox{\verb#ka#}}V,V)$, and 
the zero-set of the $\hbox{\verb#a#}\times \hbox{\verb#a#}$ minors of the  matrix representing a map in 
$\Hom(S^{\hbox{\verb#ka#}}V,V)$ are generalized eigenvectors.  Therefore, for each basis vector 
of the kernel of \verb#K#, we construct an ideal of $\hbox{\verb#a#}\times \hbox{\verb#a#}$  minors.  
Since we are interested in the common generalized eigenvectors to all basis 
vectors of the kernel of \verb#K#, we construct an ideal \verb$J$ which is generated by all 
of the {$\hbox{\verb#a'#}\times \hbox{\verb#a'#}$} minors we constructed as follows:
\verb#a'= binomial(n,ka)+1#.
\begin{verbatim}
J= ideal(0*x_0); 
for i from 0 to r-1 do 
            J= J + minors(a',RM.dd_(ka+1)|transpose(submatrix(pG,,{i}))); 
\end{verbatim}

Next we want to compute the zero-set of the the ideal \verb#J# above.  
In order to save time, we first compute the saturation of the ideal since an ideal and its saturation have the same zero-set (scheme).
\begin{verbatim}
L = decompose saturate J
\end{verbatim}
The list \verb#L# consists of linear forms which are generalized eigenvectors in the kernel of \verb#K#. 
The solutions in \verb#L# are the polar forms to those that we want.  The rest of the implementation is identical to that of the catalecticant algorithm implementation.

This example code succeeds for the following initial parameters: 
\begin{verbatim}
--n=2, (d=3, s<=3), (d=4, s<=3), (d=5, s<=7), (d=6, s<=7), 
--n=3, (d=3, s<=5}), (d=4, s<=5), (d=5, s<=11),(d=6, s<=11),
--n=4, (d=3, s<=6), (d=4, s<=6), (d=5, s<=14).
\end{verbatim}

We were able to test that over prime characteristic, all of these cases are sharp except for the case \verb#n=4, d=5#, as in this case the algorithm slowed considerably as \verb#s# grew. Note for example that when \verb#n=3#, \verb#d=5#, and \verb#s=8# there are rationality problems.

\begin{remark}[Remark on numerical methods for inexact solutions and practical issues]\label{rmk:numerical }
Many of the bounds in the examples we have presented could be improved if we allowed for irrational or complex solutions to our systems of polynomials.  In addition, we could succeed in treating the generic case if we were to use numerical eigenvector methods on the Koszul flattening and to proceed by declaring the kernel of \verb#K# to consist of those eigenvectors of \verb#K# which are associated to small eigenvalues.  We believe that the algorithms we have presented are well suited to such adaptations, however we leave this to future study.
\end{remark}

\bibliographystyle{amsalpha}
\bibliography{eigen_bibdata2}

\end{document}